%% file: RiskSharing-arxiv.tex
\documentclass[12pt,oneside,reqno]{amsart}
\usepackage{amsfonts,amsmath,amssymb,amsthm,amsxtra,url}

\usepackage[margin=1in]{geometry}
\usepackage[pdftex]{graphicx,color}

\newcommand\E{\ensuremath{\mathbb{E}}}
\renewcommand\P{\ensuremath{\mathbb{P}}}

\newcommand\R{\ensuremath{\mathbb{R}}}

\newcommand\e{\ensuremath{\mathrm{e}}}

\newcommand{\Y}{{\bf Y}}

\newcommand{\X}{{\bf X}^*}

\newcommand{\fori}{i = 1, 2, \dotsc, n}

\renewcommand{\b}{\beta}
\renewcommand{\a}{\alpha}
\newcommand{\cA}{{\mathcal A}(X)}
\newcommand{\cP}{{\mathcal P}}
\newcommand{\cC}{{\mathcal C}(X)}
\newcommand{\noi}{\noindent}
\newcommand{\cF}{{\mathcal F}}

\newcommand{\tg}{ {\tilde{g}} }

\newcommand{\la}{\lambda}
\renewcommand{\th}{\theta}

\newtheorem{theorem}{Theorem}
\newtheorem{lemma}{Lemma}
\newtheorem{proposition}{Proposition}

\newtheorem{corollary}{Corollary}
\newtheorem{definition}{Definition}

\theoremstyle{remark}
\newtheorem{remark}{Remark}
\newtheorem{example}{Example}

\DeclareMathOperator*{\esssup}{ess\,sup}

\begin{document}

\title{Optimal Risk Sharing under Distorted Probabilities}

\title[Optimal Risk Sharing under Distorted Probabilities]{Optimal Risk
Sharing under Distorted Probabilities}
\author{Michael Ludkovski$^\dag$   \and Virginia R. Young}
\address{Department of Mathematics \\ University of Michigan \\ 530 Church St. \\ Ann
Arbor, Michigan 48109 USA}
\address{Department of Statistics and Applied Probability \\ University of California \\ Santa
Barbara, California 93106 USA} \email{ludkovski@pstat.ucsb.edu, vryoung@umich.edu}
\date{\today}

\keywords{distortion risk measures, comonotonicity, risk sharing, Pareto optimal allocations}

\subjclass{91B30, 91B32, 62P05; \; JEL Classification: D81 }
\thanks{$^\dag$ Corresponding Author. Tel: 1(805)893-5634.}

\maketitle

\begin{abstract}
We study optimal risk sharing among $n$ agents endowed with distortion risk measures. Our model
includes market frictions that can either represent linear transaction costs or risk premia
charged by a clearing house for the agents.  Risk sharing under third-party constraints is also
considered. We obtain an explicit formula for Pareto optimal allocations. In particular, we
find that a stop-loss or deductible risk sharing is optimal in the case of two agents and
several common distortion functions. This extends recent result of Jouini et al.\ (2006) to the
problem with unbounded risks and market frictions.
\end{abstract}

\section{Introduction}\label{sect:intro}

Many financial problems involve transfer of risk among agents. Two noteworthy examples are
insurance markets and the general equilibrium theory of stock prices. In such problems, $n \ge
2$ agents with risky endowments (or loss exposures) $X_i$ for $i = 1, 2, \dots, n$ are
interested in devising an optimal re-allocation of their risks.  Let $X \triangleq \sum_{i =
1}^n X_i$ be the total exposure of the $n$ agents, and let $V_i$ be the subjective valuation
(preference) functional of the $i$-th agent.  Consider the collection of allocations of the
loss $X$, namely
$$
\cA \triangleq \{ {\bf Y} := (Y_1, Y_2, \dots, Y_n): X = \sum_{i = 1}^n Y_i, \; V_i(Y_i) \hbox{ finite} \}.
$$
  The risk sharing problem consists in finding an \emph{optimal}
allocation ${\bf Y}^* \in \cA$, namely an allocation such that (i) ${\bf Y}^*$ is \emph{Pareto
optimal}, that is, no agent can be made strictly better off without another agent being made
strictly worse off; and (ii) ${\bf Y}^*$ satisfies a \emph{rationality constraint}, that is,
all agents are at least as well off under ${\bf Y}^*$ as under the initial exposures ${\bf X} =
(X_1, X_2, \dots, X_n)$.  The latter feasibility constraint is motivated by the assumption that
only an irrational agent would enter into a contract that made the agent (strictly) worse off.

The key ingredient in the above problem are the preference functionals $V_i$, and accordingly
the optimal risk sharing literature has evolved as new theories of risk have been developed.
Pioneering work was carried out in the 1960s by Borch~\cite{Borch62} and Arrow~\cite{Arrow63}
who showed that deductible insurance is optimal under concave risk preferences, specifically,
when $V_i$ are represented by von Neumann-Morgenstern utility functions. Later research studied
the case of the dual theory of risk of Yaari \cite{YoungBrowne00} and Choquet expected utility
theory \cite{CDT00}. Very recently, research has focused on risk preferences given in terms of
convex risk measures \cite{FollmerSchied}. In particular, Barrieu and
El~Karoui~\cite{BarrieuKaroui05} studied optimal risk sharing under the exponential
indifference measure, while Jouini et al.~\cite{JST} analyzed the case of two agents and
convex, law-invariant risk measures. The related question of market equilibrium was addressed
in \cite{Acciaio}, \cite{BurgRusch08} and \cite{FilipovicKupper08eqm}. On a more abstract
level, Ludkovski and R\"uschendorf \cite{LR07} show that Pareto optimal allocations are
comonotone if the risk measures preserve the convex order. The latter structural result allows
for some explicit computations, as it permits direct representation of possible allocations
through the pooling functions.

A simultaneous strand of the literature has been addressing extensions of the basic insurance
problem that take into account market frictions. For example, the fundamental problem of
adverse selection was initiated by Rothschild and Stiglitz \cite{RS76} and later further
discussed in \cite{YoungBrowne00}. The effect of transaction costs on optimal contracts was
first considered by Raviv~\cite{Raviv79}. Other possible externalities are summarized in the
survey articles of Gerber~\cite{Gerber78} and Aase~\cite{Aase02}. Many markets also impose
constraints on possible risk transfers. Often, only a limited set of risk instruments is a
priori given, so that risk sharing must belong to the span of available contracts (as studied
by Filipovic and Kupper~\cite{FilipovicKupper08}). Alternatively, the amount of risk transfer
is limited by regulator authorities; for instance in the classical insurance problem the
insurer may be able to take on only part of the total risk due to risk capital regulations. The
latter problem, which we call risk sharing under constraints, introduces effectively $n+1$
players into the model, namely $n$ original participants, plus the additional regulator that
imposes limits on allowable risk exposures of each participant. The special case of
Value-at-Risk constraints was recently analyzed in Bernard and Tian~\cite{BernardTian08}.

This article extends previous results in these two directions by studying optimal risk sharing
in the context of distortion risk measures, transaction costs and/or third-party constraints.
Distortion risk measures lie at the junction of actuarial and financial applications, being
related both to the dual theory of risk and coherent risk measures. The transaction costs in
our model have a dual nature and can either represent genuine transaction fees arising due to
verification, accounting and other inter-agent costs, or the risk-loaded premium charged by the
insurer. For the constraints, we consider a general set of restrictions given in terms of
distortion risk measures.
%


Our main result, namely Theorem \ref{thm:Po}, shows that in all of the above cases, the optimal
risk allocation consists of a collection or ``ladder'' of deductible contracts. This result can
be interpreted as an economic justification for the  tranche contracts one observes in
practice, in particular, in credit and reinsurance markets. Moreover, using the quantile
representation of distortion risk measures we are able to explicitly characterize Pareto
optimal contracts under transaction costs and/or constraints. In turn, this allows us to
present several completely worked-out examples of optimal risk sharing under some common risk
measures, such as Average Value-at-Risk.

In terms of related literature, Theorem \ref{thm:Po} is an extension of the results of Jouini
et al.~\cite{JST} to the multi-agent case with transaction costs and constraints. Compared to
their abstract approach based on convex duality an inf-convolution, our method is more
elementary and direct and provides a clearer insight into the problem structure. On a more
general note, this paper aims to underscore the usefulness of distortion risk measures that
have been arguably under-appreciated by the financial/mathematical economics community
\cite{Dhaene06}. In contrast to the classical expected utility theory, this new framework is
driven by two factors. First, it postulates cash-equivariant preferences that are appealing
based on the normative observation that guaranteed cash payments should not affect risk
attitudes. Secondly, distortion risk measures attempt to mirror business practices where
various Value-at-Risk (VaR) methodologies have emerged as the tool of choice. In particular,
Average Value-at-Risk (AVaR) has been gaining practitioner acceptance and also happens to be a
canonical example of our model.

This paper is organized as follows:  In Section \ref{sect:model}, we define the setting in
which the $n$ agents seek a Pareto optimal risk exchange.  In Section \ref{sect:pareto}, we
obtain the class of Pareto optimal risk exchanges in our model. This is then generalized to the
constrained setting in Section \ref{sect:constraint}. We focus on the case of two agents in
Section \ref{sect:n=2}, while interpreting one agent as an insurer and another as a buyer of
insurance. In this simplified setup we present fully solved examples, including examples with
explicitly computable deductibles. In Section \ref{sect:buyer}, we provide another illustration
of our results by considering a single-agent minimization by a buyer of insurance who faces a
regulator constraint on the possible indemnity contracts. Section \ref{sect:conclusion}
concludes the paper.

\section{Model for Risk Sharing}\label{sect:model}

\subsection{Distorted Probabilities}

Consider the collection of a.s.-finite random variables $\cP = \{Y: \P[-\infty < Y < \infty] =
1\}$ on a probability space $(\Omega, \cF, \P)$. As usual, we denote by $L^\infty \subset \cP$
($L^1 \subset \cP$) the collection of all a.s.\ bounded (respectively integrable) random
variables.

\begin{definition}
Two random variables $Y$ and $Z \in \cP$ are said to be {\rm comonotone} if
\begin{equation}
(Y(\omega_1) - Y(\omega_2)) (Z(\omega_1) - Z(\omega_2)) \ge 0,
\end{equation}

$\P(d\omega_1) \times \P(d\omega_2)$-almost surely. In other words, $Y$ and $Z$ move together.
\end{definition}

An equivalent definition of comonotonicity is that there exists a random variable $V \in \cP$
and non-decreasing functions $f_Y$ and $f_Z$ such that $Y = f_Y(V)$ and $Z = f_Z(V)$ almost
surely \cite{Denneberg94}.  Another equivalent definition is that there exist non-decreasing
functions $h_Y$ and $h_Z$ such that $h_Y(x) + h_Z(x) = x$, $Y = h_Y(Y + Z)$, and $Z = h_Z(Y +
Z)$ almost surely.

\begin{definition}
A function $H : \cP \to \R$ is called a {\rm law-invariant, comonotone, monetary risk measure} $($or {\rm distortion risk measure}$)$ if $H$ satisfies the following five properties:
\begin{enumerate}
\item[(a)] $H(Y)$ depends only on the law of $Y \in \cP$.

\item[(b)] $H$ is monotone in the natural order of $\cP$.

\item[(c)] $H$ is cash equivariant: $H(Y + a) = H(Y) + a$ for any $a \in \R$.

\item[(d)] $H$ is subadditive in general and additive for comonotone risks:  For $Y, Z \in \cP$,
\begin{align}
H(Y + Z) \leq H(Y) + H(Z),
\end{align}
with equality for any $Y, Z$ comonotone.

\item[(e)] $H$ is continuous:  For $Y \in \cP$,
\begin{subequations}
\begin{align}
&\lim_{d \rightarrow -\infty} H[ \max(Y, d)] = H(Y),\\
&\lim_{d \rightarrow 0^+} H[ \max(Y - d, 0)] = H(Y), \hbox{ if } Y \ge 0, \\
&\lim_{d \rightarrow \infty} H[ \min(Y, d)] = H(Y).
\end{align}
\end{subequations}

\end{enumerate}
\end{definition}

The above axioms are justified by basic economic principles as applied to insurance; see
\cite{Dhaene06,JST,WYP97}.  Because we are interested in risk sharing, cash equivariance is a
desirable property because receiving fixed payments (at least within a reasonable range) should
not affect attitudes towards risk. The comonotone additivity property represents  inability to
diversify risks that always move in the same direction. The continuity property (e) is for
technical reasons, although it was shown by \cite{Jouini05} that viewing $H$ as a map on
$L^\infty(\P)$, (e) is automatically implied by (a)-(d).

Denote by $S_Y$ the (decumulative) distribution function of $Y$, that is, $S_Y(t) = \P(Y > t)$,
and by $S^{-1}_Y$ the (pseudo-)inverse of $S_Y$, which is unique up to a countable set
\cite{Denneberg94}.  For concreteness, take $S_Y^{-1}(p) = \sup \{ t: S_Y(t) > p \}$.  The
inverse $S^{-1}_Y$ thus defined is right continuous; if one were to desire left continuity,
then replace $>$ with $\ge$.

We recall that any distortion risk measure admits the following representation, which
essentially follows from Greco's representation theorem \cite{Greco82}:

\begin{theorem} {\rm (\cite{Denneberg94}, \cite[Appendix A]{WYP97})}
Let $H$ be a distortion risk measure. Then, there exists a non-decreasing, concave function $g:[0,1] \to [0,1]$ such that $g(0) = 0$, $g(1) =1$, and
\begin{align}\label{eq:2.1}
H(Y) & = \int Y \, d(g \circ \P) = \int_0^1 S^{-1}_Y(p) \, dg(p) \\ \notag & = \int_{-\infty}^0 \left( g[S_Y(t)] - 1 \right) \, dt + \int_0^\infty g[S_Y(t)] \, dt , \quad \forall Y \in \cP.
\end{align}
\end{theorem}

\medskip

We write $H_g$ for $H$ when we want to specify the particular function $g$ in \eqref{eq:2.1}.  The function $g$ is called a {\it distortion} because it modifies, or distorts, the tail probability $S_Y$.   Observe if $g(p) = p$, then $H_g(Y) = \E Y$.  For this reason, $H_g$ is also referred to as an expectation with respect to a distorted probability.  Note that at this stage we allow $H_g$ to take $\pm \infty$ as a value.

We assume that each agent orders random variables in $\cP$ by using a distortion risk measure
$H_g$, where $Y$ is preferred to (that is, less risky than) $Z$ by the agent if $H_g(Y) \le
H_g(Z)$, and we pursue this topic in the next section.  For more background on such risk
measures $H_g$, see Yaari~\cite{Yaari87} who discusses evaluating random variables in a theory
of risk that is dual to expected utility. Two noteworthy examples of distortion risk measures
are (1) the Average Value-at-Risk at level $1-\alpha^{-1}$ (AVaR) obtained by taking $g(p)=
\min(\alpha p, 1)$ for some $\alpha > 1$ and (2) the proportional hazards transform $g(p) =
p^c$ for some $0 < c < 1$.

\begin{remark} It has been shown \cite{Cherny06,Kusuoka01} that any distortion risk measure is a weighted average of the AVaR. Namely, define $AVaR_\alpha(Y)$ as above. Then, any comonotone law-invariant  coherent risk measure on $\cP$ can be written as
$$ H(Y) = \int_0^1 AVaR_\alpha(Y) \mu(d\alpha), $$
for some probability measure $\mu$ on $[0,1]$. For this reason, \cite{Cherny06} calls a distortion risk measure Weighted VaR.
\end{remark}

\begin{remark}
Since a distortion risk measure is a special case of a coherent risk measure, one can also
obtain a representation of $H$ in terms of penalized expectations, $H(Y) = \sup_{Q \in
\mathcal{D}} \E_Q[Y]$, for the set $\mathcal{D}$ of probability measures called the {\it core}
of $g \circ \P$, and $Y \in L^1$ \cite[Proposition 10.3]{Denneberg94}. For more results in this
direction see \cite{FollmerSchied}.
\end{remark}

\begin{definition}
$Y$ is said to {\rm precede} $($or be {\rm preferred} to$)$ $Z$ in \emph{convex order} if $\int_0^q S^{-1}_{Y}(p) \, dp \le \int_0^q S^{-1}_{Z}(p) \, dp$ for all $q \in [0, 1]$ with equality at $q = 1$.  We write $Y \le_{cx} Z$.
\end{definition}

Note that convex order is equivalent to ordering with respect to second stochastic dominance
with equal means \cite{RS70,RS71,RS72}.  For later use, recall that $H_g$ satisfies the
following properties for $Y \in \cP$ (see \cite{WangYoung98}):

\begin{enumerate}
\item[(a)] Positive homogeneity:  If $a \ge 0$, then $H_g(aY) = a H_g(Y)$.  Note that positive homogeneity and subadditivity imply that $H$ is convex, that is, $H(\lambda Y + (1 - \lambda) Z) \le \lambda H(Y) + (1 - \lambda) H(Z)$ for all $\lambda \in (0, 1)$.

\item[(b)] Duality:  $H_g(-Y) = - H_\tg(Y)$, in which $\tg$ is the dual distortion of $g$ given
by $\tg(p) = 1 - g(1 - p)$ for $p \in [0, 1]$.  Since $g$ is concave, $\tg$ is convex. The dual
$H_\tg$ can be thought of as a monetary utility function that measures attitudes towards wealth
levels; see \cite{Jouini05}.

\item[(c)] Convex ordering: Because $g$ is concave, $H_g$ preserves $\le_{cx}$, that is, if $Y
\le_{cx} Z$ then $H_g(Y) \le H_g(Z)$. In particular, because $\E Y \le_{cx} Y$, then $\E Y =
H_g(\E Y) \le H_g(Y)$.

\item[(d)] Non-excessive loading:  $H(Y) \le \esssup  Y$.
\end{enumerate}


\subsection{Economic Objective}

Suppose agent $i$ faces a random loss $X_i$ before any risk exchange for $\fori$.  If the
collection of agents trades the original allocation $\bf X$ for the allocation ${\bf Y} \in
\cA$, then the random loss or payout, including transaction costs, of agent $i$ becomes
\begin{equation}\label{eq:terminal}
Z_i = Y_i + (a_i + b_i Y_i + c_i \E Y_i) = (1 + b_i)Y_i + a_i + c_i \E Y_i.
\end{equation}

\noi The additive factor $a_i \ge 0$ is a fixed cost associated with transferring the risk $X_i$ to the coalition of agents (or to a central clearing house); for example, $a_i$ could be the premium that the agent pays to the coalition to eliminate the risk $X_i$.  The multiplicative factor $b_i \ge 0$ represents costs associated with the actual size of the random loss $Y_i$, for example, investigative costs that could increase proportionally with the size of the loss.  The factor $c_i \in \R$ represents costs that reflect the {\it expected} size of the payout $Y_i$, for example, hiring claim administrators; $c_i$ is also net of any
premium that the agent {\it receives} in exchange for accepting the risk $Y_i$, if the premium equals $(1 + \theta) \E Y_i$ as in \cite{Arrow63}.  In fact, we might wish to say that $c_i = -(1 + \theta)$, that is, {\it all} of this part of the cost function arises from premium received.  We explore this in examples later in the paper, as well as at the end of this section.

Agent $i$, for $i = 1, 2, \dots, n$, seeks to minimize $H_{g_i}(Z_i)$ for some concave distortion function $g_i$.  Note that minimizing
\begin{equation}\label{eq:Htg}
H_{g_i}(Z_i) = H_{g_i}((1 + b_i)Y_i + a_i + c_i \E Y_i) = (1 + b_i)
H_{g_i}(Y_i) + a_i + c_i \E Y_i
\end{equation}
\noi is equivalent to minimizing
\begin{equation}\label{eq:Hg}
V_i(Y_i) : = (1 + b_i) H_{g_i}(Y_i) + c_i \E Y_i.
\end{equation}

In light of this recasting of agent $i$'s goal, a Pareto optimal risk exchange is defined as follows:

\begin{definition}\label{def:Po}
${\bf X}^* \in \cA$ is called a {\rm Pareto optimal} risk exchange or allocation if whenever
there exists an allocation ${\bf Y} \in \cA$ such that $V_i(Y_i) \le V_i(X^*_i)$ for all
$\fori$, then $V_i(Y_i) = V_i(X^*_i)$ for all $\fori$.
\end{definition}

\noi In other words, there is no way to make any agent (strictly) better off without making another agent (strictly) worse off.

We assume that the initial allocation carries finite risk, that is, $H_{g_i}(X_i)$ is finite for $\fori$.  Therefore, there exists at least one allocation $\bf Y$, namely $\bf X$ itself, such that $V_i(Y_i)$ is finite for all $\fori$.

We end this section by discussing the rationality constraint mentioned in the Introduction.  In order that the allocation ${\bf Y} \in \cA$ be feasible (regardless of whether it is Pareto optimal), it must be true that each agent is at least as well off under $\bf Y$ as under the original allocation $\bf X$.  That is, the following inequality must hold for each $\fori$:  $H_{g_i}(X_i) \ge V_i(Y_i)$.  We assume that the set of feasible allocations in $\cA$ is non-empty.

When first presenting the cost function $a_i + b_i Y_i + c_i \E Y_i$ in connection with equation (\ref{eq:terminal}), we proposed that one might wish to consider the last term as representing premium received in exchange for accepting the risk $Y_i$.  In that case, write the premium as $-c_i \E Y_i = (1 + \theta) \E Y_i$, so that the rationality constraint becomes
\begin{equation}\label{eq:prem}
(1 + \theta) \E Y_i \ge a_i + (1 + b_i) H_{g_i}(Y_i) - H_{g_i}(X_i).
\end{equation}

\noi One can interpret the left-hand side of inequality (\ref{eq:prem}) as the minimum premium that agent $i$ is willing to accept for replacing $X_i$ with $Y_i$.  Therefore, the rationality constraint holds in this case if the premium received is at least as great as the risk-adjusted cost, as measured by the right-hand side of (\ref{eq:prem}).

\section{Pareto Optimal Allocations}\label{sect:pareto}

To describe the Pareto optimal allocations, we begin with a series of lemmas.  In the first
lemma, we show that if the $1 + b_i + c_i$'s are of different signs or if one of them is zero
and the other is non-zero, then no Pareto optimal allocation exists.

\begin{lemma}\label{lem:diffsign}
Suppose there exist $i, j = 1, 2, \dots, n$ such that $1+b_i + c_i \neq 0$ and $(1 + b_i +
c_i)(1 + b_j + c_j) \le 0,$ then no Pareto optimal allocation in $\cA$ exists.
\end{lemma}

\begin{proof}
Without loss of generality, suppose that $1 + b_1 + c_1 < 0$ and $1 + b_2 + c_2 \ge 0$.
Consider any $\Y \in \cA$. Then, ${\bf Z} = (Y_1 + 1, Y_2 - 1, Y_3, \dots, Y_n)$ is a strict
improvement on $\Y$ because $V_1(Z_1) = V_1(Y_1) + (1 + b_1 + c_1) < V_1(Y_1)$ and $V_2(Z_2) =
V_2(Y_2) - (1 + b_2 + c_2) \le V_2(Y_2)$.  Thus, there exists no Pareto optimal allocation in
$\cA$. \qed
\end{proof}

For the present, we skip the case in which all $1 + b_i + c_i = 0$ for $\fori$; we consider it
more fully for the case of $n = 2$ in Section \ref{sect:n=2}.  The next lemma is
straightforward, but we include its proof for completeness.

\begin{lemma}\label{lem:easy}
If ${\bf X}^* = (X^*_1, X^*_2, \dotsc, X^*_n) \in \cA$ is Pareto optimal, then so is $(X^*_1, X^*_2, \dots, X^*_j + \beta, \dots, X^*_k - \beta, \dots, X^*_n) \in \cA$ for any $\beta \in \R$ and any $j, k = 1, 2, \dotsc, n$.
\end{lemma}

\begin{proof}
Let ${\bf X}^* = (X^*_1, X^*_2, \dots, X^*_n) \in \cA$ be Pareto optimal.  Suppose $\Y \in \cA$
is such that $V_j(Y_j) \le V_j(X^*_j + \beta)$, $V_k(Y_k) \le V_k(X^*_k - \beta)$, and
$V_i(Y_i) \le V_i(X^*_i)$ for $i \ne j, k$.  We want to show that equality holds in each case.
Inequality $V_j(Y_j) \le V_j(X^*_j + \beta)$ implies that $V_j(Y_j) \le V_j(X^*_j) + (1 + b_j +
c_j) \beta$, from which it follows that $V_j(Y_j - \beta) \le V_j(X^*_j)$.  Similarly,
$V_k(Y_k) \le V_k(X^*_k - \beta)$ implies that $V_k(Y_k + \beta) \le V_k(X^*_k)$.  Note that
the allocation $\Y'$ defined by $Y'_j = Y_j - \beta$, $Y'_k = Y_k + \beta$, and $Y'_i = Y_i$
for $i \ne j, k$ is in $\cA$.  Therefore, by the Pareto optimality of ${\bf X}^*$ we have
$V_j(Y_j - \beta) = V_j(X^*_j)$, $V_k(Y_k + \beta) = V_k(X^*_k)$, and $V_i(Y_i) = V_i(X^*_i)$
for $i \ne j, k$, from which it follows that $V_j(Y_j) = V_j(X^*_j + \beta)$, $V_k(Y_k) =
V_k(X^*_k - \beta)$, and $V_i(Y_i) = V_i(X^*_i)$ for $i \ne j, k$.  Hence, $(X^*_1, X^*_2,
\dots, X^*_j + \beta, \dots, X^*_k - \beta, \dots, X^*_n)$ is Pareto optimal. \qed
\end{proof}

It follows from Lemma \ref{lem:easy} that without loss of generality, we can assume that a Pareto optimal allocation assigns the loss 0 to each of the $n$ agents when the total loss $X$ is 0.  If this particular Pareto optimal allocation does not satisfy the rationality constraint in inequality (\ref{eq:prem}), then we can modify the allocation by constants (that sum to zero) so that the rationality constraint is satisfied.  (Recall that we assume that the set of feasible allocations is non-empty, so there exist such constants.)

Consider the mapping $F: \cA \rightarrow \R^n$ given by $F(\Y) = (V_1(Y_1), V_2(Y_2), \dots, V_n(Y_n))$.  We can partially order the points in $\R^n$ as follows:

\begin{definition}
For ${\bf x}, {\bf y} \in \R^n$, we write ${\bf x} \le {\bf y}$ if $x_i \le y_i$ for $\fori$.
\end{definition}

\noi The next lemma, whose proof is immediate from the definition of Pareto optimality in
Definition \ref{def:Po}, shows that the Pareto optimal points in $\cA$ correspond to the
minimal points in the image of $F$ in $\R^n$.
\begin{lemma}\label{lem:min}
If $\X \in \cA$ is Pareto optimal, then $F(\X) \in im(F)$ is minimal.  Conversely, if ${\bf x}
\in im(F)$ is minimal, then there exists $\X \in \cA$ with $F(\X) = {\bf x}$, such that $\X$ is
a Pareto optimal allocation.
\end{lemma}

%

We next use Lemmas \ref{lem:easy} and \ref{lem:min} to characterize the set of Pareto optimal
allocations when we view them as points in $\R^n$ via the mapping $F$.

\begin{theorem}\label{thm:F}
Suppose $(1 + b_i + c_i)(1 + b_j + c_j) > 0$ for all $i, j = 1, 2, \dots, n$.  Then, the image of the set of Pareto optimal allocations in $\cA$ under the mapping $F$ is a hyperplane in $\R^n$ given by
\begin{equation}\label{eq:hyper}
\left\{ {\bf x} \in \R^n: \sum_{i=1}^n (V_i(X^*_i) - x_i)/ \left(1 + b_i + c_i \right) = 0 \right\},
\end{equation}
in which ${\bf X}^* \in \cA$ is any Pareto optimal allocation.  Furthermore, one obtains such a Pareto optimal allocation $\X$ by minimizing
\begin{equation}\label{eq:min}
\sum_{i=1}^n V_i(Y_i)/ \big|1 + b_i + c_i \big|
\end{equation}
over $\Y \in \cA$.
\end{theorem}

\begin{proof}
We begin by showing that if $\X \in \cA$ minimizes the expression in (\ref{eq:min}), then $\X$ is Pareto optimal.  Suppose that $\Y \in \cA$ is such that $V_i(Y_i) \le V_i(X^*_i)$ for $\fori$.  Then, $\sum_{i=1}^n V_i(Y_i)/ \big|1 + b_i + c_i \big| \le \sum_{i=1}^n V_i(X^*_i)/ \big|1 + b_i + c_i \big|$, from which it follows that $\sum_{i=1}^n V_i(Y_i)/ \big|1 + b_i + c_i \big| = \sum_{i=1}^n V_i(X^*_i)/ \big|1 + b_i + c_i \big|$ because $\X$ minimizes (\ref{eq:min}).  Therefore, $V_i(Y_i) = V_i(X^*_i)$ for $\fori$, and $\X$ is Pareto optimal.

Next, suppose ${\bf x} \in \R^n$ satisfies the equation of the hyperplane (\ref{eq:hyper}) for
some Pareto optimal allocation $\X \in \cA$.  Define $\beta_i := (x_i - V_i(X^*_i))/(1 + b_i +
c_i)$ for $\fori$; then, $\sum_{i=1}^n \beta_i = 0$.  Define ${\bf \hat X}^* := (X^*_1 +
\beta_1, X^*_2 + \beta_2, \dots, X^*_n + \beta_n) \in \cA$.   By the same argument as in the
proof of Lemma \ref{lem:easy}, one can show that ${\bf \hat X}^*$ is Pareto optimal.  Finally,
$F({\bf \hat X}^*) = (V_1(X^*_1 + \beta_1), V_2(X^*_2 + \beta_2), \dots, V_n(X^*_n + \beta_n))
= F(\X) + (\beta_1(1 + b_1 + c_1), \beta_2(1 + b_2 + c_2), \dots, \beta_n(1 + b_n + c_n)) =
F(\X) + (x_1 - V_1(X^*_1), x_2 - V_2(X^*_2), \dots, x_n - V_n(X^*_n)) = {\bf x}$.  Thus, (any)
${\bf x}$ in \eqref{eq:hyper} is an image of a Pareto optimal allocation in $\cA$ via the
mapping $F$. As an aside, note that all elements of the hyperplane (\ref{eq:hyper}) give the
same minimum value in the expression (\ref{eq:min}).

To complete the proof, we need to show that the hyperplane (\ref{eq:hyper}) gives us all the
Pareto optimal allocations.  Suppose not; suppose that there is a Pareto optimal allocation
$\Y^* \in \cA$ that is mapped to a point not on the hyperplane (\ref{eq:hyper}).  Then, by the
argument in the above paragraph, any point ${\bf y} \in \R^n$ that satisfies $\sum_{i=1}^n
(V_i(Y^*_i) - y_i)/ \left(1 + b_i + c_i \right) = 0$ is the image of a Pareto optimal
allocation.  Thus, we have two parallel hyperplanes both purporting to be the image (under the
mapping $F$) of Pareto optimal allocations in $\cA$.  By Lemma \ref{lem:min}, only one of these
hyperplanes will be minimal, a contradiction.  Thus, the Pareto optimal allocations in $\cA$
correspond to points in the hyperplane (\ref{eq:hyper}). \qed \end{proof}

To describe Pareto optimal allocations corresponding to points in the hyperplane (\ref{eq:hyper}), it is easier to consider comonotone allocations.

\begin{definition}
An allocation ${\bf Y} \in \cA$ is called comonotone if $Y_i$ and $X$ are comonotone for $\fori$.
\end{definition}

Note that if $\bf Y$ is a comonotone allocation then any two $Y_i$ and $Y_j$ are also  pairwise
comonotone.  Ludkovski and R\"uschendorf \cite[Proposition 1]{LR07} shows that for $V_i$
preserving the convex order, any integrable non-comonotone allocation ${\bf X} \in \cA$, $X_i
\in L^1(\P)$ is dominated by some comonotone $ \bf X^*$, $V_i( X^*_i) \le V_i( X_i)$, $\fori$.
This result is essentially based on the comonotone $\le_{cx}$-improvement result of Landsberger
and Meilijson~\cite{LM94}. Note that the requirement $X_i \in L^1$ is automatically satisfied
since we already assume that $\E X_i \le H_{g_i}(X_i) < \infty$.  Thus, Pareto optimal
allocations are comonotone.

For a comonotone allocation ${\bf X} = (f_1(X), f_2(X), \dots, f_n(X))$,
Denneberg~\cite[Proposition 4.5]{Denneberg94} shows that the functions $f_i$ are continuous on
$supp(X)$ for $\fori$. Moreover, he shows that $f_i$ may be extended to continuous functions on
the entire real line such that $\sum_{i=1}^n f_i(x) = x$ for all $x \in \R$.  It follows that
we can restrict our attention to finding Pareto optimal allocations in
\begin{multline}\label{eq:4.8}
\cC \triangleq \{(f_1(X), f_2(X), \dots, f_n(X)) \in \cA\colon \\   f_i \text{ cont., non-decreasing}, \; \sum_{i = 1}^n f_i(x) = x \hbox{ for } x \in \R \}.
\end{multline}
Comonotonicity implies that an optimal risk allocation necessarily satisfies the mutuality principle, whereby the share of each agent depends only on the total risk $X$. We now use the above results to explicitly characterize the Pareto optimal allocations.
\begin{theorem}\label{thm:Po}
Suppose $(1 + b_i + c_i)(1 + b_j + c_j) > 0$ for all $i, j = 1, 2, \dots, n$.  Then, $\X = (f^*_1(X), f^*_2(X), \dots, f^*_n(X)) \in \cC$ is a Pareto optimal allocation if and only if
\begin{equation}\label{eq:optI}
\sum_{i \in {\mathcal I}} (f^*_i)'(t) = 1 \hbox{ for } \; {\mathcal I} = \hbox{\rm argmin}_{k = 1, 2, \dots, n} \; \frac{(1 + b_k)g_k(S_X(t)) + c_k S_X(t)}{ \big| 1 + b_k + c_k \big|},
\end{equation}
\noi and $(f^*_i)'(t) = 0$ otherwise.
\end{theorem}

\begin{proof}
From Theorem \ref{thm:F} and \cite{LR07}, we know that Pareto optimal allocations correspond to minimizers $\X \in \cC$  of the expression in (\ref{eq:min}).  As discussed after the proof of Lemma \ref{lem:easy} without loss of generality, suppose that the Pareto optimal allocation $\X = (f^*_1(X), f^*_2(X), \dots, f^*_n(X))$ is such that $f^*_i(0) = 0$ for $\fori$.

Suppose $Y = f(X)$ for a continuous, non-decreasing real-valued function $f$ on $\R_+$ with $f(0) = 0$; then,
\begin{equation}\label{eq:int-by-parts}
\begin{split}
(1 + b) & H_g(Y) + c \E Y = (1 + b) \int_0^1 S^{-1}_{f(X)} (p) \, dg(p) + c \int_0^1 S^{-1}_{f(X)} (p) \, d(p)  \\
& = (1 + b) \int_0^1 f \left[ S^{-1}_X(p) \right] \, dg(p) + c \int_0^1 f \left[ S^{-1}_X(p) \right] \, d(p) \\
& = (1 + b) \int_0^\infty g \left[ S_X(t) \right] \, df(t) + c \int_0^\infty S_X(t) \, df(t) \\
& = \int_0^\infty \left[ (1 + b) g + c \right] (S_X(t)) \, df(t),
\end{split}
\end{equation}
\noi in which the function $(1 + b) g + c$ is defined on $[0, 1]$ by $[(1 + b) g + c](p) = (1 + b) g(p) + cp$.  Thus, minimizing expression (\ref{eq:min}) is equivalent to minimizing
\begin{equation}\label{eq:tranches}
\sum_{i=1}^n \int_0^\infty \frac{ \left[ (1 + b_i) g_i + c_i \right] (S_X(t))}{ \big| 1 + b_i + c_i \big| } \, df_i(t),
\end{equation}
\noi which is minimized by setting $\sum_{i \in {\mathcal I}} (f^*_i)'(t) = 1$ for \hfill
\break ${\mathcal I} = \hbox{\rm argmin}_{k = 1, 2, \dots, n} \; \left\{(1 + b_k)g_k(S_X(t)) +
c_k S_X(t) \right\}/ \big| 1 + b_k + c_k \big|$, and by setting $(f^*_i)'(t) = 0$ otherwise.
\qed \end{proof}

The above theorem implies that under a Pareto optimal allocation, the risk sharing consists of
``tranches''  where the risk of each tranche is entirely borne by one agent (ignoring equality
in the argmin). As expression \eqref{eq:optI} shows, the optimal allocation $Y^*_i$ of the
$i$-th agent consists of a series of laddered European options on the total risk $X$. Hence,
agent $i$ assumes total responsibility for risk levels where $f^*_i( S^{-1}_X(t) ) = 1$, and
receives full insurance otherwise. Such risk sharing arrangements are observed in practice in
credit derivatives, where the total risk $X$ represents a bond portfolio subject to default
risk and the corresponding risk is allocated via credit tranches. These credit tranches can be
viewed as optimal insurance contracts for a set of representative investors with varying risk
measures.

\begin{remark} The problem considered in this section has a long history in the context of reinsurers determining the best way to allocate  risk among them.   Borch~\cite{Borch62} shows that if the reinsurers seek to maximize their expected utility of wealth, then the allocation is related to their absolute risk aversions, in which the absolute risk aversion associated with a utility function $u$ is $-u''/u'$.
B\"uhlmann~\cite{Buhlmann80,Buhlmann84} extends Borch's work by developing premium rules
associated with such risk sharing. The connection between second order stochastic dominance and
optimality of deductible insurance was already noted in \cite{heerwaarden} and
\cite{Gollier96}.
\end{remark}

\begin{remark}\label{keyRemark}
Theorem \ref{thm:F} and the reduction to comonotone allocations are key steps in our argument
since they dramatically simplify the structure of Pareto optimal allocations.  Note that the
only property used in the proof of Theorem \ref{thm:F}  was the cash equivariance of the
corresponding risk measures, while the only property used in relation to the comonotonicity
improvement of Proposition 1 in Ludkovski and R\"uschendorf~\cite{LR07} was consistency of $H$
and $\le_{cx}$. On the other hand, B\"auerle and M\"uller~\cite{BauerleMuller06} show that any
law-invariant convex risk measure, subject to a mild continuity requirement, is consistent with
the convex order $\le_{cx}$.  We, therefore, hypothesize that the conclusion of Theorem
\ref{thm:Po} will hold for arbitrary law-invariant convex risk measures. This conjecture would
further extend  the setting of Jouini et al.~\cite{JST}.
\end{remark}


\section{Constrained Risk Sharing}\label{sect:constraint}

We next consider the related situation for which the risk sharing is subject to regulation.
This may arise, for example, in an insurance setting where the risk transfer from buyer to
insurer is controlled by a government regulator, or in a financial setting where the party
taking on risk is subject to a risk management framework, such as Basel II.

The effect of such regulation is to impose further constraints upon some of the $Y_i$'s in
\eqref{eq:min}. This of course modifies the resulting Pareto optimal allocations since some of
the possible optima become infeasible under the constraint. A similar model was studied by
Bernard and Tian~\cite{BernardTian08} under the assumption of a VaR constraint.   In our
framework where we work with distortion risk measures, we instead postulate constraints of the
form
$$
H_{h_i}(Y_i) \le B_i, \qquad \fori,
$$
in which $H_{h_i}$ is the regulator's (convex) risk measure on the final risk transfer amount
$Y_i$, and $B_i$ is the corresponding risk threshold for agent $i$.

We modify the set of allocations $\cA$ to account for these constraints.  Define the set of {\it constrained} allocations by
$$
{\mathcal A}^c(X) \triangleq \{ {\bf Y} := (Y_1, Y_2, \dots, Y_n): X = \sum_{i = 1}^n Y_i, \; V_i(Y_i) \hbox{ finite}, \; H_{h_i}(Y_i) \le B_i \}.
$$
We assume that the set of feasible allocations in ${\mathcal A}^c(X)$ is non-empty.  Analogous to Definition \ref{def:Po}, ${\bf X}^* \in {\mathcal A}^c(X)$ is a {\it constrained} Pareto optimal allocation  if whenever there exists an allocation ${\bf Y} \in {\mathcal A}^c(X)$ such that $V_i(Y_i) \le V_i(X^*_i)$ for all $\fori$, then $V_i(Y_i) = V_i(X^*_i)$ for all $\fori$.


The next lemma shows that as in Section \ref{sect:pareto} for unconstrained Pareto optimal
allocations, without loss of generality we can restrict our attention to constrained Pareto
optimal allocations that are comonotone.

\begin{lemma}\label{lem:constraint-co}
If ${\bf Y} \in {\mathcal A}^c(X)$, then there exists ${\bf Y}' \in {\cC} \cap {\mathcal
A}^c(X)$ that improves it in the partial ordering of Section \ref{sect:pareto}.
\end{lemma}

\begin{proof}
Ludkovski and R\"uschendorf~\cite[Proposition 1]{LR07} show that given an arbitrary allocation
${\bf Y} \in {\mathcal A}^c(X) \subset \cA$, there is a comonotone improvement in the
stochastic convex order ${\bf Y}' \in \cC$, that is, $Y_i' \le_{cx} Y_i$  for $\fori$.
Therefore, the allocation ${\bf Y}'$ improves $\bf Y$ in the partial ordering of Section
\ref{sect:pareto} because $V_i$ preserves the convex order for $\fori$.  Moreover, because
$H_{h_i}$ also preserves the convex order, it follows that $H_{h_i}(Y'_i) \le H_{h_i}(Y_i)$ for
$\fori$ and ${\bf Y}' \in {\mathcal A}^c(X)$ is still feasible. Thus, ${\bf Y}' \in \cC \cap
{\mathcal A}^c(X)$. \qed
\end{proof}

Note that if the constraining risk measure is not convex, then optimal allocations might not be
comonotone. For instance, a VaR constraint at level $\alpha\%$ corresponds to the non-concave
distortion function $h(p)=1_{\{p > \alpha\}}$. Such $H_h$ is not consistent with the
$\le_{cx}$-order, and therefore Lemma \ref{lem:constraint-co} does not apply. Indeed, as
explicitly shown by Bernard and Tian~\cite{BernardTian08}, the resulting optimal allocation
might fail to be comonotone.

By using Lemma \ref{lem:constraint-co}, we reduce the constrained problem to the same situation as
in Theorem \ref{thm:Po}.
\begin{theorem}\label{thm:constraint-Po}
The optimal risk allocation for the constrained problem is obtained by finding minimizers of
\begin{equation}\label{eq:constrained-min}
 \sum_{i=1}^n \int_0^\infty \frac{ \left[ (1 + b_i) g_i + \lambda_i h_i + c_i\right]
(S_X(t))}{ \big| 1 + b_i + c_i +\lambda_i\big| } \, df_i(t),
\end{equation}
in which $\la_i \ge 0$ is a Lagrange multiplier for the $i$-th constraint, for $\fori$.
\end{theorem}
\noindent It follows from Theorem \ref{thm:constraint-Po}, that we, again, will obtain a
ladder-like optimal contract structure, similar to the tranches in \eqref{eq:tranches}. Many
cases are possible with respect to which of the $\lambda_i$'s are positive (that is, the
respective constraint binds) versus zero.  In particular, a variety of degeneracies might arise
if several constraints bind simultaneously.  Instead of considering all these cases for an
arbitrary $n$, in Sections \ref{sect:ex-con} and \ref{sect:buyer} we focus on a simple example
with two agents and one constraint, a setting already taken up in \cite{BernardTian08}.

\section{The Special Case of $n = 2$ Agents}\label{sect:n=2}

In this section, we specialize our results to the case for which we have two agents.  Suppose
an individual (agent 2) is facing an insurable random loss $X_2 = X$ and wants to buy insurance
$f(X)$ for all or part of the loss $X$ from an insurer (agent 1 with $X_1 = 0$). In this case,
our problem amounts to finding a Pareto optimal allocation $(f^*(X), X - f^*(X))$, in which
$f^*(X)$ is the insurer's share of the risk $X$, and $X - f^*(X)$ is the amount of the risk
retained by the individual. Arrow~\cite{Arrow63} showed that if the premium equals $(1 +
\theta) \E f(X)$ with $\theta > 0$ and if the individual seeks to maximize his or her expected
utility of wealth, then $f^*(X)$ is deductible coverage (that is, $f^*(X)$ is given
functionally by $f^*(x) = (x - d)_+$ for some $d \ge 0$, in which $x$ is a specific value of
the random loss $X$).  One could view this risk exchange as Pareto optimal if the insurer's
goal were to maximize its expected profits (among other possible criteria).  For more recent
work in the area of optimal insurance, see Promislow and Young~\cite{PromislowYoung05} who
extended the work of Arrow to other premium rules and optimality criteria.

We first examine the case for which $1 + b_1 + c_1 = 0 = 1 + b_2 + c_2$.  Then, we consider the case for which $(1 + b_1 + c_1)(1 + b_2 + c_2) > 0$.

\subsection{$1 + b_1 + c_1 = 0 = 1 + b_2 + c_2$}

In this case, we have $c_1 = -(1 + b_1)$ and $c_2 = -(1 + b_2)$.  It follows from arguments similar to those in Section \ref{sect:pareto} that the Pareto optimal risk exchanges are given as the minimizers over $\Y \in \cC$ of the following expression as $\lambda_1$ and $\lambda_2$ range over the non-negative reals:
\begin{equation}\label{eq:argmin}
\lambda_1 (1 + b_1) \left[ H_{g_1}(Y_1) - \E Y_1 \right] + \lambda_2 (1 + b_2) \left[ H_{g_2}(Y_2) - \E Y_2 \right],
\end{equation}

\noi with at least one of $\lambda_1$ and $\lambda_2$ strictly positive.  Without loss of generality, suppose $\lambda_1 > 0$.  Also, note that $Y_2 = X - Y_1$; let $f(X)$ denote $Y_1$.  Then, the Pareto optimal risk exchanges are the minimizers over real-valued, continuous, non-decreasing functions $f$, with $H_{g_i}(f(X))$ finite for $i = 1, 2$, of the following expression as $\delta$ ranges over the non-negative reals:
\begin{equation}\label{eq:argmin2}
\left[ H_{g_1}(f(X)) - \E f(X) \right] + \delta \left[ \E f(X) - H_{g_2}(f(X)) \right].
\end{equation}

Without loss of generality, we can assume that $f(0) = 0$; otherwise, define $\hat f$ by $\hat f(x) - f(0)$ and note that $H_{g_i}(\hat f(X)) - \E \hat f(X) = H_{g_i}(f(X) - f(0)) - \E (f(X) - f(0)) = H_{g_i}(f(X)) - \E f(X)$ for $i = 1, 2$.

By following the argument of Theorem \ref{thm:Po}, the derivative of the optimal function $f^*$ is given by
\begin{equation}\label{eq:4.1}
(f^*)'(t) =
\begin{cases}
1, & \hbox{if } g_1(S_X(t)) - S_X(t) < \delta  \left[ g_2(S_X(t)) - S_X(t) \right]; \\
\beta, & \hbox{if } g_1(S_X(t)) - S_X(t) = \delta  \left[ g_2(S_X(t)) - S_X(t) \right]; \\
0,  & \hbox{otherwise},
\end{cases}
\end{equation}

\noi in which $\beta \in [0, 1]$ is arbitrary.  If we interpret $g(S_X(t)) - S_X(t)$ as the marginal cost of adding more risk (except for the factor of $1 + b$), then $f^*$ increases if the marginal cost for the insurer is less than the marginal cost for the buyer adjusted by the factor $\delta \ge 0$.

Note that if $0 \le \delta_1 < \delta_2$, then $f^*_{\delta_1} \le f^*_{\delta_2}$, in which $f^*_{\delta_i}$ corresponds to the minimizer of (\ref{eq:argmin2}) for $\delta = \delta_i$, $i = 1, 2$.  In other words, as the weight given to buyer's risk preference increases, then the insurer assumes more of the risk.

In the special case for which $\delta = 0$, we seek to minimize $H_{g_1}(f(X)) - \E f(X)$ which is greater than or equal to $0$ because $g_1$ is concave.  Thus, $H_{g_1}(f(X)) - \E f(X)$ is minimized by $f^* \equiv r$ for any constant $r$.  If $g_1$ is strictly concave, then this expression is minimized {\it uniquely} (up to an additive constant) by $f^* \equiv 0$.  If $g_1$ is not strictly concave, then for illustrative purposes, suppose $g_1$ is given by AVaR, specifically $g_1(p) = \min(\alpha p, 1)$ for some $\alpha > 1$.  Then, for $X \sim Bernoulli(q)$ for some $q > 1/\alpha$, the function $f$ given by $f(0) = r$ and $f(1) = r + 1$ is such that $H_{g_1}(f(X)) - \E f(X) = 0$ for any $r \in \R$.  That is, if $g_1$ is not strictly concave, then the minimizer of $H_{g_1}(f(X)) - \E f(X)$ is not necessarily unique.

In general, if the distortions are not strictly concave, then it is possible that $g_1(S_X(t)) - S_X(t) = \delta  \left[ g_2(S_X(t)) - S_X(t) \right]$ on a set of positive measure, in which case, $f^*$ will not be unique.

We leave the case for which $1 + b_1 + c_1 = 0 = 1 + b_2 + c_2$ because as the reader will see in the next section, the conclusions that we could draw further from equation \eqref{eq:4.1} are similar to the ones we will draw from equation (\ref{eq:4.2}) below.

\subsection{$(1 + b_1 + c_1)(1 + b_2 + c_2) > 0$}\label{sec:ex}

Let $f(X)$ be the random indemnity that the insurer (agent 1) pays to the buyer (agent 2) in exchange for a premium of $(1 + \theta) \E f(X)$ for some $\theta > 0$, with $f(X)$ and $X - f(X)$ comonotone.

For concreteness, in the notation of this paper, set $a_1 = 0,$ $b_1 > 0,$ $c_1 = -(1 +
\theta)$ and $a_2 = (1 + \theta) \E X$,  $b_2 = 0,$ $c_2 = -(1 + \theta)$. Thus, the condition $(1 + b_1 + c_1)(1 + b_2 + c_2) > 0$ is equivalent to $b_1 < \theta$.

Under these values for the parameters, the rationality constraint for the insurer in $(\ref{eq:prem})$ becomes
\begin{equation}\label{eq:insurer}
(1 + \theta) \E f(X) \ge (1 + b_1) H_{g_1}(f(X));
\end{equation}

\noi that is, the insurer is willing to enter into a contract for which the premium $(1 + \theta) \E f(X)$ is at least as great as the risk-adjusted cost, as measured by $(1 + b_1) H_{g_1}(f(X))$.  The rationality constraint for the buyer becomes \begin{equation}\label{eq:buyer}
H_{g_2}(f(X)) \ge (1 + \theta) \E f(X);
\end{equation}

\noi that is, the risk-adjusted benefit for the buyer from receiving $f(X)$  is greater than the cost $(1 + \theta) \E f(X)$.

It is reasonable to assume that the buyer is ``more risk averse'' than the insurer in the sense that the buyer's distortion function is a concave transformation of the insurer's, or equivalently, $g_2 \ge g_1$. Theorem \ref{thm:Po} then implies that the optimal function $f^*$ is given by
\begin{equation}\label{eq:4.2}
(f^*)'(t) =
\begin{cases}
1, & \hbox{if } g_1(S_X(t)) - S_X(t) < \frac{\theta - b_1}{ \theta(1 + b_1)}  \left[ g_2(S_X(t)) - S_X(t) \right]; \\
\beta, & \hbox{if } g_1(S_X(t)) - S_X(t) = \frac{\theta - b_1}{  \theta(1 + b_1)} \left[ g_2(S_X(t)) - S_X(t) \right]; \\
0,  & \hbox{otherwise}.
\end{cases}
\end{equation}

\noi in which $\beta \in [0, 1]$ is arbitrary.  The function $f^*$ in equation (\ref{eq:4.2}) is similar in form to the one given in (\ref{eq:4.1}), with the arbitrary $\delta \ge 0$ replaced by the fixed $0 < (\theta - b_1)/(\theta(1 + b_1)) < 1$.

From the expression in (\ref{eq:4.2}), we can deduce several conclusions.  Because $(\theta - b_1)/\theta$ increases as $\theta$ increases, the optimal insurance $f^*$ increases as the proportional risk loading $\theta$ increases.  Also, because $(\theta - b_1)/(1 + b_1)$ decreases as $b_1 < \theta$ increases, the optimal insurance $f^*$ decreases as the insurer's cost $b_1$ increases.  This makes sense because if the proportional cost of the insurer increases, as measured by $b_1$, then the insurer is willing to sell less insurance to the buyer.

If $g_2$ is replaced by a concave distortion $\hat g_2 \ge g_2$, then $f^*$ increases because $g_2(S_X(t)) - S_X(t)$ increases.  In other words, as the buyer of insurance  becomes more risk averse, then the buyer is willing to purchase more insurance at a given price.

We have the following proposition that tells us when the optimal insurance is deductible insurance.  We omit its proof because it is a straightforward application of the expression in (\ref{eq:4.2}).  Recall from the discussion following Lemma \ref{lem:easy} that without loss of generality, we can assume that $f^*(0) = 0$, and we do so in this proposition.

\begin{proposition}\label{prop:ded}
If $(g_1(p) - p)/(g_2(p) - p)$ increases for $p \in (0, 1)$, then deductible insurance is optimal, that is,
\begin{equation}\label{eq:dedcov}
f^*(x) = (x - d)_+
\end{equation}
is optimal with the deductible $d$ given by
\begin{equation}
d = \inf \left\{t: \frac{g_1(S_X(t)) - S_X(t)}{ g_2(S_X(t)) - S_X(t)} \le \frac{\theta - b_1}{ \theta(1 + b_1)}  \right\}.
\end{equation}
If no such $d$ exists, then $f^* \equiv 0$.
\end{proposition}

\noi Note that if $(g_1(p) - p)/(g_2(p) - p)$ increases for $p \in (0, 1)$, then $(g_1(S_X(t)) - S_X(t))/(g_2(S_X(t)) - S_X(t))$ decreases for $t \ge 0$.

Proposition \ref{prop:ded} is a generalization of Proposition 3.2 in Jouini et al.~\cite{JST}
who also obtained deductible insurance in the context of law-invariant convex risk measures. In
contrast to the proof presented here, their non-constructive method relies on convex duality
and only applies in the setting of $L^\infty(\P)$.

We have three corollaries to Proposition \ref{prop:ded} for special cases of distortion
functions.  First, we consider the proportional hazards transform; then, we consider AVaR;
finally, we consider the dual power distortion.  We omit their proofs because they follow
directly from showing that $(g_1(p) - p)/(g_2(p) - p)$ increases on $(0, 1)$.

\begin{corollary}
If $g_i(p) = p^{c_i}$ for $0 < c_2 < c_1 < 1,$ then deductible insurance is optimal.
\end{corollary}

\noi Moreover, for the proportional hazards transform, $(g_1(p) - p)/(g_2(p) - p)$ increases from $0$ to $(1 - c_1)/(1 - c_2) < 1$.  Therefore, if $(1 - c_1)/(1 - c_2) < (\theta - b_1)/(\theta(1 + b_1))$ then full coverage is optimal, which occurs when $b_1$ is small enough.  However, if $\theta$ is large, then the rationality constraint in inequality (\ref{eq:buyer}) might not hold, so full coverage (even though optimal) might not be feasible.  In such cases, we can subtract a fixed amount $a > 0$ from the coverage to make it feasible by the buyer, thereby effectively lowering the benefit and the premium.   Finally, note that as $c_2$ decreases (that is, as the buyer becomes more risk averse), the ratio $(g_1(S_X(t)) - S_X(t))/(g_2(S_X(t)) - S_X(t))$ decreases for a given value of $t \ge 0$, which implies that the deductible decreases (that is, the optimal coverage increases).

\begin{corollary}
If $g_i(p) = \min (\alpha_i p, 1)$ for $1 < \alpha_1 < \alpha_2,$ then deductible insurance is optimal.
\end{corollary}

\noi For the AVaR distortion, $(g_1(p) - p)/(g_2(p) - p)$ increases from $(\alpha_1 -
1)/(\alpha_2 - 1)$ to 1.  If $(\alpha_1 - 1)/(\alpha_2 - 1) > (\theta - b_1)/(\theta(1 +
b_1))$, then zero coverage is optimal.  If $b_1 > 0$ and if $S_X(0) = 1$, then full coverage is
never optimal.

\begin{corollary}
If $g_i(p) = 1 - (1 - p)^{d_i}$ for $1 < d_1 < d_2,$ then deductible insurance is optimal.
\end{corollary}

\noi The dual power distortion is so named because it is the dual to the proportional hazards transform.  For this distortion, $(g_1(p) - p)/(g_2(p) - p)$ increases from $(d_1 - 1)/(d_2 - 1)$ to $\infty$.  Thus, if $(d_1 - 1)/(d_2 - 1) > (\theta - b_1)/(\theta(1 + b_1))$, then zero coverage is optimal.  If $S_X(0) = 1$, then full coverage is never optimal.

We end this section with two examples in which we show that deductible coverage as defined in
the narrow sense of equation (\ref{eq:dedcov}) is not necessarily optimal.

\begin{example}\label{ex:1}
Define the distortions $g_1$ and $g_2$ on $[0, 1]$ by
\begin{equation}
g_1(p) =
\begin{cases}
\frac{9}{ 8} p, &0 \le p \le \frac{1}{2}, \\
\frac{7}{ 8} p + \frac{1}{8}, &\frac{1 }{2} < p \le 1;
\end{cases}
\end{equation}

\noi and
\begin{equation}
g_2(p) =
\begin{cases}
\frac{4}{3} p, &0 \le p \le \frac{1}{4}, \\
p + \frac{1}{12}, & \frac{1}{4} < p \le \frac{3}{4}, \\
\frac{2}{3} p + \frac{1}{3}, &\frac{3}{4} < p \le 1.
\end{cases}
\end{equation}

\noi  If $X \sim Exp(1)$, $\theta = 1$, and $b_1 = 1/3$, then one can show that optimal insurance $f^*_1$ satisfies
\begin{equation}
(f^*_1)'(t) =
\begin{cases}
1, &0 \le t < \ln \frac{3}{2}, \\
0, & \ln \frac{3}{2} \le t < \ln 3, \\
1, & t \ge \ln 3.
\end{cases}
\end{equation}

\noi In other words, optimal insurance $f^*_1$ exhibits full coverage up to $\ln(3/2)$ followed by no additional coverage until $\ln 3$, after which the coverage is full at the margin. Specifically, $f^*_1$ is given by
\begin{equation}
(f^*_1)(t) =
\begin{cases}
t, &0 \le t < \ln {3 \over 2}, \\
\ln {3 \over 2}, & \ln {3 \over 2} \le t < \ln 3, \\
t + \ln {1 \over 2}, & t \ge \ln 3.
\end{cases}
\end{equation}
\end{example}

\begin{example}
Define the distortions $g_1$ and $g_2$ on $[0, 1]$ by
\begin{equation}
g_1(p) =
\begin{cases}
{4 \over 3} p, &0 \le p \le {1 \over 4}, \\
p + {1 \over 12}, & {1 \over 4} < p \le {3 \over 4}, \\
{2 \over 3} p + {1 \over 3}, &{3 \over 4} < p \le 1.
\end{cases}
\end{equation}

\noi and
\begin{equation}
g_2(p) =
\begin{cases}
{3 \over 2} p, &0 \le p \le {1 \over 2}, \\
{1 \over 2} p + {1 \over 2}, &{1 \over 2} < p \le 1;
\end{cases}
\end{equation}

\noi If $X \sim Exp(1)$, $\theta = 1$, and $b_1 = 1/3$, then one can show that the optimal insurance $f^*$ paid by the insurer is given by $f^*(t) = t - f^*_1(t)$ for $t \ge 0$, in which $f^*_1$ is the optimal insurance in Example \ref{ex:1}.  In other words, optimal insurance in this case exhibits a deductible of $\ln(3/2)$ with a maximum limit, or maximum payout, of $\ln 2$.
\end{example}


\subsection{Examples with Constraints}\label{sect:ex-con}

Regulators of insurance often put constraints on insurance contracts that insurers are allowed
to provide in the market.  To illustrate the effect of constraints on the form of the indemnity
contract $f$, we include two simple examples.  In both these examples, we follow the model for
two agents with $b_1 > 0$, $b_2 = 0$, and $c_1 = c_2 = -(1+\theta)$. Let
\begin{equation*}
\left\{ \begin{aligned} g_1(p) & = \min( \a_1 p, 1), \\
g_2(p) & = \min( \a_2 p, 1), \\ h_1(p) & = \min( \beta p, 1). \end{aligned} \right.
\end{equation*}
Agent 1 is the insurer with the AVaR distortion function $g_1$ that faces a regulator
constraint based on the $H_{h_1}$ risk measure; agent 2 is the buyer with the AVaR distortion
function $g_2$.

\begin{example}\label{ex:3} In this example, suppose that $\a_2 > \beta > \a_1 > 1$; that is, the
buyer is the most risk averse with the insurer being the least risk averse and the regulator
somewhere in between. The relevant terms in the sum \eqref{eq:constrained-min} are given by
\begin{equation}\label{eq:distortion-ex1} \left\{ \begin{aligned}
Q_1(p) & = [(1+b_1)\min(\a_1 p,1) - (1+\th)p + \lambda \min(\beta p, 1) ]/| b_1 +\la -\th| \\
Q_2(p) &= [\min(\a_2 p, 1) - (1+\th)p]/\th.
\end{aligned}\right.
\end{equation}
\noi By Theorem \ref{thm:constraint-Po}, for a given Lagrange multiplier $\la \ge 0$, the
optimal contract satisfies $(f^\la)'( S_X(t)) = 1$ if $Q_1(p)<Q_2(p)$ and $(f^\la)'(S_X(t))=0$
otherwise.

In the following, we assume that $\th > \la + b_1$, so the transaction costs are \emph{large}.
The risk functions $Q_1$ and $Q_2$ are illustrated in Figure \ref{fig:constrained}. We note
that for
large $p \sim 1$, $Q_1(p) \ge Q_2(p)$ and moreover, the two piecewise linear functions cross at
most once on $(0,1)$. More precisely, if $\a_2 > (1+\th) + \frac{[(1+b_1)\a_1 - (1+\th) + \la
\beta]\th}{\th-b_1-\la}$, then for small $p \sim 0$, $Q_1(p) < Q_2(p)$, and $Q_1$ and $Q_2$
have exactly one crossing point $0 < p^* < 1$.  Thus, the optimal contract in that case is
deductible insurance $f^\la(x) = (x-d)_+$, as the insurer covers large risks (small $p$) and
the buyer takes on small risks. If $\a_2$ is smaller than the above threshold, then $Q_1(p) >
Q_2(p)$ for all $p \in (0,1)$, and it is optimal to have zero insurance $f^\la \equiv 0$ (note
that zero insurance implies $\la =0$ as the constraint is necessarily non-binding).

The two (finite) possibilities for the deductible level $d$ (with $S_X(d)$ corresponding to the
unique crossing point of $Q_1$ and $Q_2$) are illustrated in Figure \ref{fig:constrained}. The
left panel of Figure \ref{fig:constrained} shows Case (a), whereby
\begin{equation}\label{eq:p-star2}
 S_X(d)= p^*_2 = \frac{\la(1+\th) - \th + b_1}{(1+\th)(b_1+\la) - (1+b_1)\a_1 \th}.
\end{equation}
The necessary and sufficient condition for Case (a) to occur is $1/\b < p^*_2 < 1/\a_1$, which
is equivalent to
$$
-b_1 < \lambda < \min\left( \th - b_1, \frac{(\th-b_1)\beta +
(1+\th)b_1-(1+b_1)\a_1\th}{(1+\th)(\beta-1)} \right).
$$
It is possible that the upper bound is negative which implies that case (a) cannot occur as
$\la$ is non-negative by construction.

Otherwise, we are in Case (b) shown on the right panel of Figure \ref{fig:constrained}, where
\begin{equation}
S_X(d) = p^*_1 = \frac{\th - (b_1 + \la)}{\th(1+b_1)\a_1 - b_1(1+\th) + \la[\th\beta - (1+\th)]}.
\end{equation}
Case (b) requires that $1/\a_2 < p^*_1 < 1/\b$, or
$$
\frac{b_1(1+\th)-\th(1+b_1)\a_1+(\th-b_1)\beta}{(\beta-1)(1+\th)} < \la < \frac{b_1(1+\th)-\th(1+b_1)\a_1+(\th-b_1)\a_2}{(\a_2 - 1) + \th (\beta - 1)}.
$$

\end{example}
\begin{example}\label{ex:4}
We keep the above notation but now suppose that $\beta > \a_2 > \a_1 >1$; that is, the
regulator is the most risk averse, the buyer is moderately risk averse, and again, the insurer
is the least risk averse.

Let $\lambda \ge 0$ be a Lagrange multiplier for this problem. We continue to assume $\th > b_1
+ \la$. The risk functions to compare are the same as in \eqref{eq:distortion-ex1} but their
relation has changed, as illustrated in the bottom panel of Figure \ref{fig:constrained}. In
particular, it is now possible that $Q_1$ and $Q_2$ cross twice in the interior of $(0,1)$, so
that the optimal contract may be a capped deductible.  Specifically, in the latter case
$$
f^\la(x) = (x-d_1)_+ \wedge d_2, \qquad\text{ where   } d_2 = S_X^{-1}\left(p^*_2\right)
$$
from \eqref{eq:p-star2} and
$$
d_1 = S_X^{-1}\left( 
\frac{\la\th}{(1+\th)(b_1+\la)-(1+b_1)\a_1\th+\a_2(\th-b_1-\la)} \right),
$$
subject to the feasibility constraints $1/\beta < S_X(d_1) < 1/\a_2$ and $1/\a_2 < S_X(d_2) <
1/\a_1$. Translating these into constraints for $\lambda$ we find that
$$
\max\left(-b_1, \frac{(1+\th)b_1}{\beta\th + \a_2 -(1+\th)} \right) < \la < \min\left( \th
- b_1,  \frac{(1+\th)b_1}{\a_2(1+\th)-(1+\th)} \right).
$$
This situation is illustrated in the bottom panel of Figure
\ref{fig:constrained}. Note that because $\a_2 > \a_1$, if there were no constraints, then the
optimal insurance would be deductible insurance.

Note that with such a contract, the regulator's risk level takes the form $H_{h_1}(f^\la(X)) =
S_X(d_2) - S_X(d_1)$ (since $1/\beta < S_X(d_1)<S_X(d_2)$). For instance, taking $\a_1 = 1.1,
\a_2=1.5, \beta = 2$, $\th=1.2, b_1=0.3$, and $\la = 0.18$, we obtain $S_X(d_1) = 0.5143$ and
$S_X(d_2) = 0.7636$, so that the insurer only covers the $23-48$th percentiles of the risk.
Since the constraint is binding, $B = H_{h_1}( f^\la(X)) = 0.249$, and looking back we can
interpret this as saying that the insurer is allowed to cover at most 24.9\% of the risk.
Observe that even though the regulator is having a lot of impact on the optimal contract (the
constraint is binding), the risk aversion of the insurer himself $\a_1$ still plays a role in
the shape of the insurance contract.
\end{example}

\begin{figure}[ht]
\begin{tabular*}{\textwidth}{lr}
\begin{minipage}{2.5in}\resizebox{!}{5cm}{
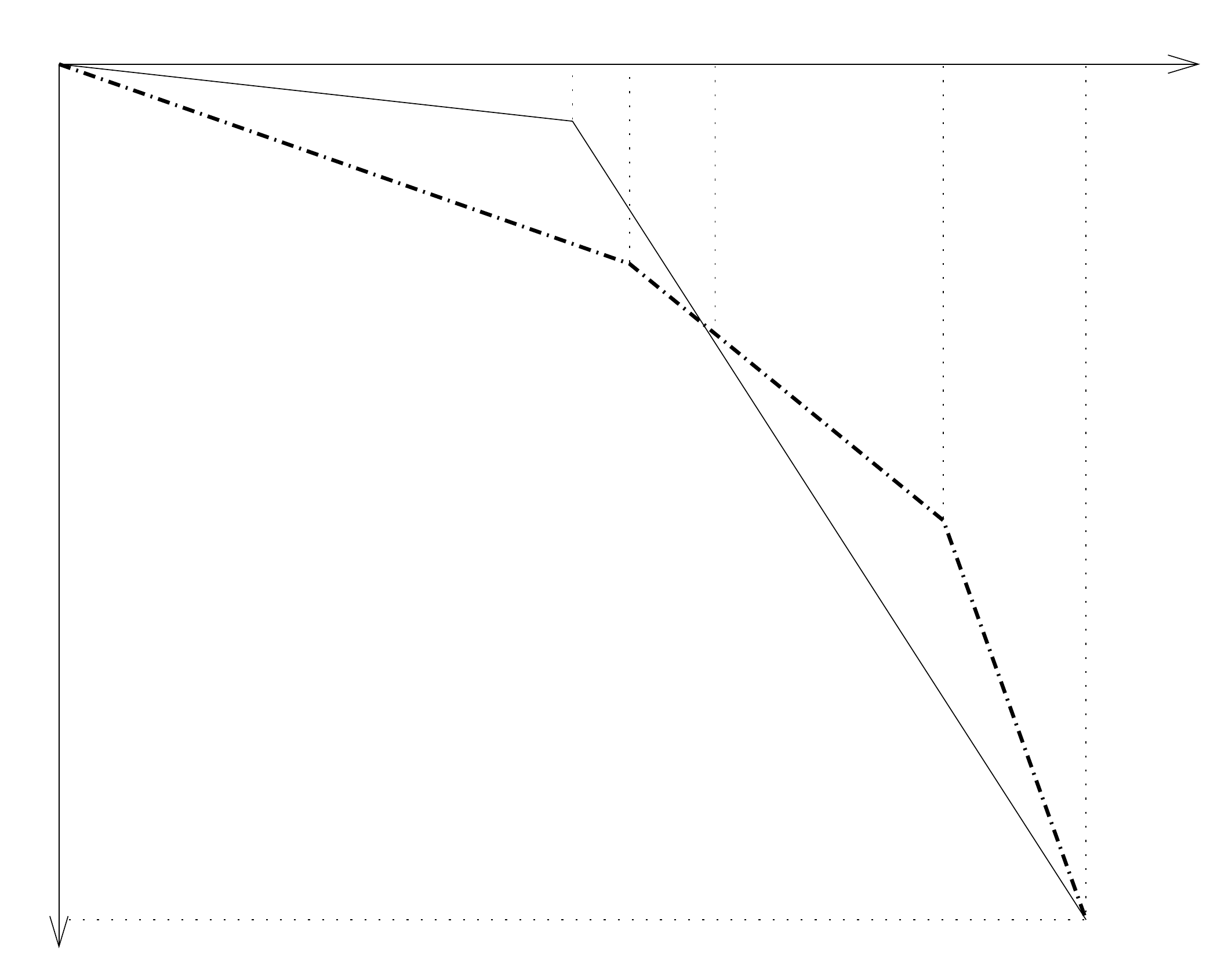}
\end{minipage} &
\begin{minipage}{2.5in}\resizebox{!}{5cm}{
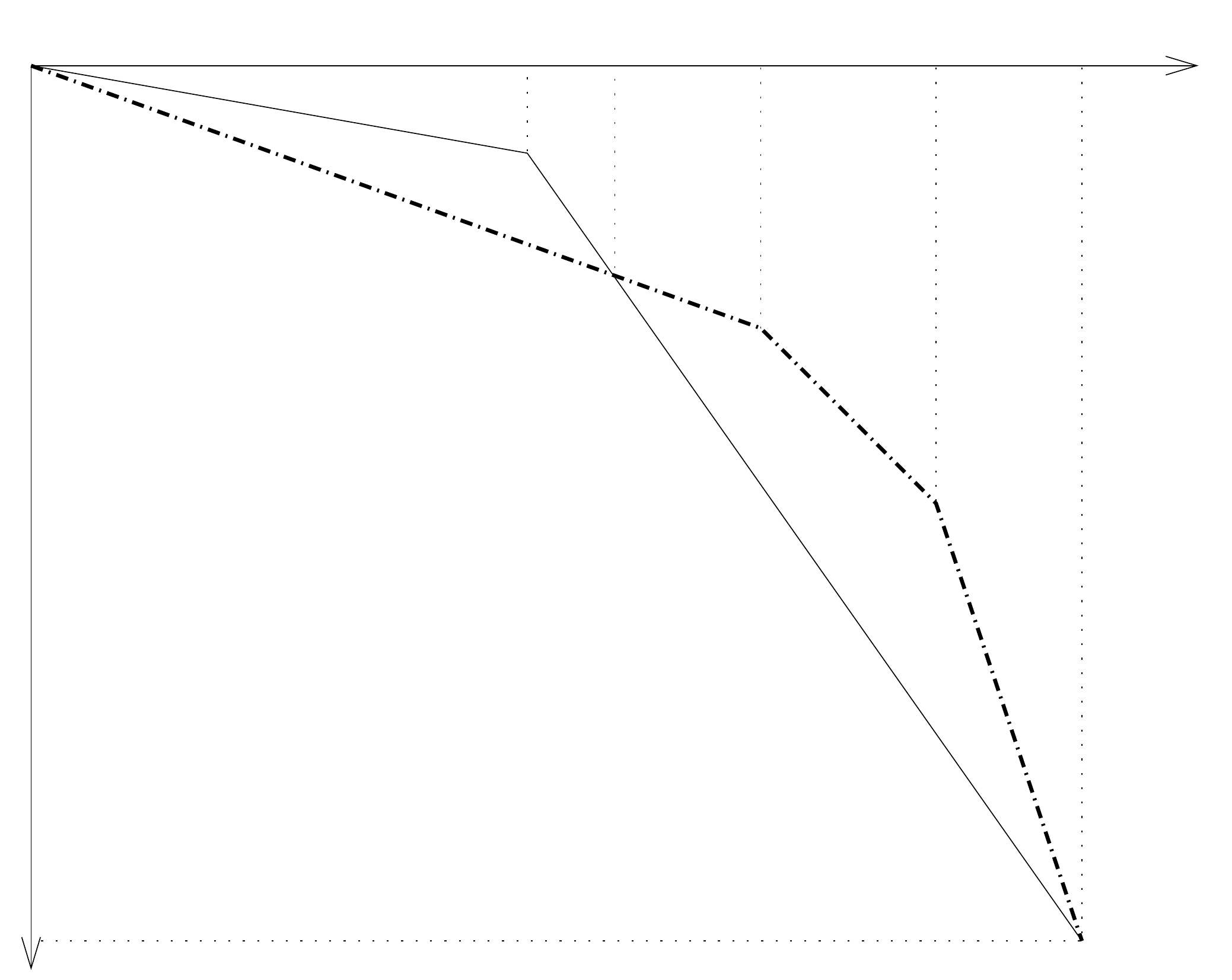}
\end{minipage} \\
\begin{minipage}{2.5in}\resizebox{!}{5cm}{
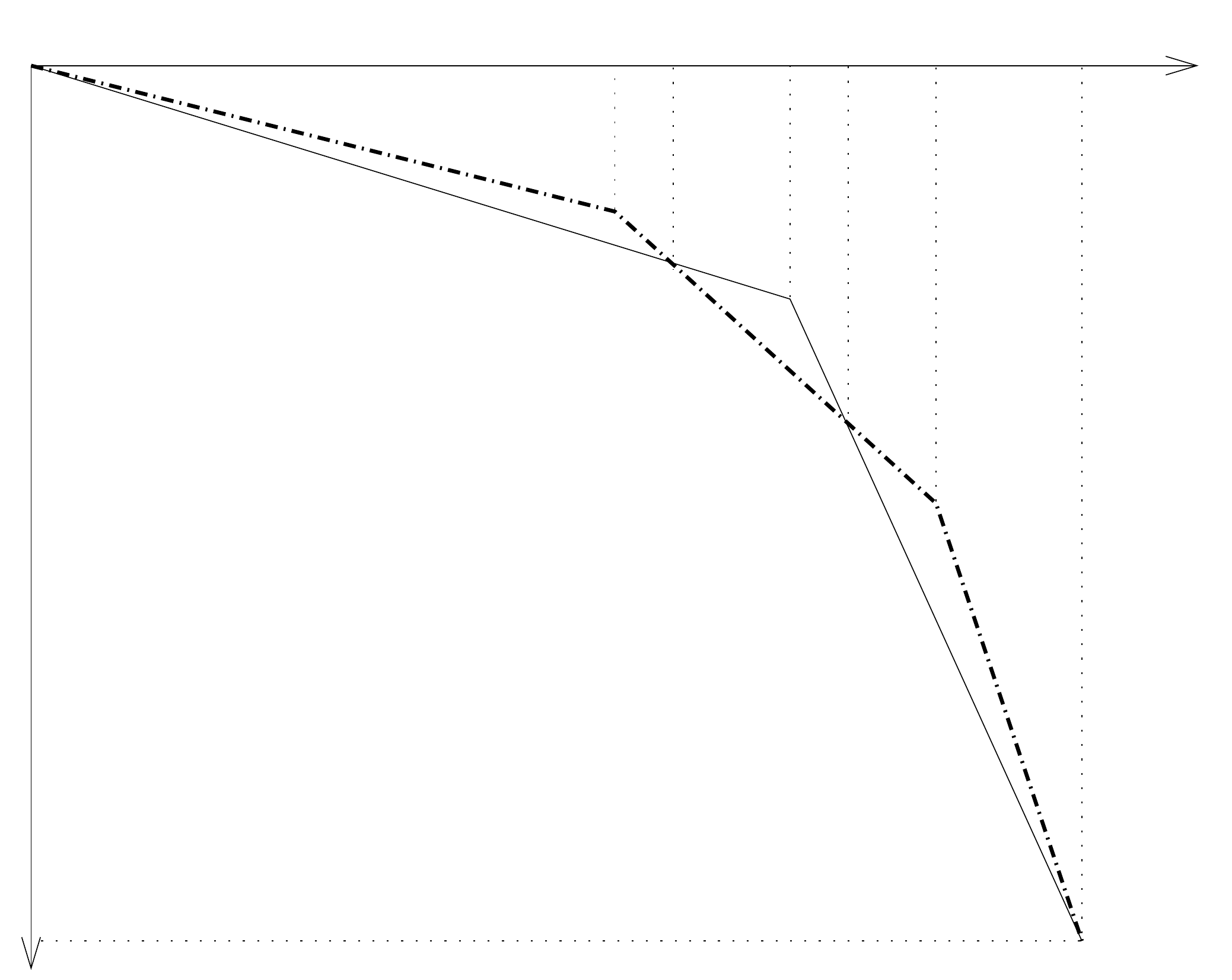}
\end{minipage} &
\end{tabular*}
\caption{Risk functions of Examples \ref{ex:3} and \ref{ex:4}. The dashed line represents
$Q_1(p)  = [(1+b_1)\min(\a_1 p,1) - (1+\th)p + \lambda \min(\beta p, 1) ]/| b_1 +\la -\th|$, and
the solid line is $Q_2(p) = [\min(\a_2 p, 1) - (1+\th)p]/\th$. In this example, $\theta
> \lambda+b_1$, so we have $Q_1(1)=Q_2(1)=-1$. Note that both functions are piecewise
linear. The crossing points correspond to the tranche levels of optimal contracts. The top two
panels are for Example \ref{ex:3} (Case (a) on the left, Case (b) on the right), and the bottom
panel is for Example \ref{ex:4}. \label{fig:constrained}}
\end{figure}

\section{Minimizing the Risk of the Buyer subject to a Constraint}\label{sect:buyer}

To further explore the implications of constrained risk sharing, we consider a slightly
different example in which the buyer is the only minimizing agent. This is the usual insurance
setting whereby the insurer offers a menu of contracts and the buyer selects the one most
suited to her needs. Thus, the optimization is from the buyer's point of view; the insurer's
risk preferences enter the problem through the insurance price.

Assume that the buyer's risk-adjusted loss after obtaining insurance is $(1+b) H_g(X-f(X)) + (1 + \theta) \E f(X)$,  in which the first term
represents the residual risk and the second term represents the insurance premium. The insurer
himself is constrained by regulators to $H_h(f(X)) \le B$, so that only a limited amount of
risk may be transferred.  We ignore the desires of the insurer and focus on minimizing the
risk-adjusted loss of the buyer subject to this constraint.  Then, we seek to find a
non-decreasing $f^*$ that minimizes
\begin{equation}\label{eq:buyer-con}
(1+b)H_g(X-f(X)) + (1+\theta) \E f(X),
\end{equation}
subject to the regulatory constraint
\begin{equation}\label{eq:constr}
H_h(f(X)) \le B,
\end{equation}
for some $B > 0$. The following proposition is a direct counterpart of Theorem
\ref{thm:constraint-Po}.

\begin{theorem}\label{thm:single-agent-constraint}
An insurance contract $f^*$ that minimizes \eqref{eq:buyer-con} subject to \eqref{eq:constr}
is determined by
\begin{equation}\label{eq:constraint-tranche}
 (f^*)'(t)  = \left\{ \begin{aligned}
 1, && \text{ if } & (1+b) g( S_X(t)) > (1+\theta) S_X(t) + \lambda h ( S_X(t)), \\
0, && \text{ if } & (1+b) g( S_X(t)) \le (1+\theta) S_X(t) + \lambda h ( S_X(t)).
\end{aligned} \right.
\end{equation}
Furthermore, either $\la = 0$ or $\la > 0$, with the latter implying that \eqref{eq:constr}
holds with equality, from which we can determine $\la$.
\end{theorem}

\begin{proof}
Fix $\la \ge 0$. Proceeding as in \eqref{eq:int-by-parts}, we have
\begin{align*}
& (1+b) H_g( X- f^{\la}(X) ) + (1+\theta) \E Y + \lambda( H_h(f^{\la}(X)) - B)  \\
& \qquad = \int_0^\infty [ -(1+b) g + (1+\theta) + \lambda h]( S_X(t) ) \, df^{\la}(t) + Const.
\end{align*}
Thus, to minimize \eqref{eq:buyer-con} we should set $(f^{\la})'(t) = 0$ when the integrand is
positive, and $f'(t)= 1$ when the integrand is negative, which is equivalent to
\eqref{eq:constraint-tranche}.  To find $\la$, we solve for $\lambda \int_0^\infty h(S_X(t)) \,
df^{\lambda}(t) = B$.\qed
\end{proof}

To be concrete, take $g(p) = \min( \a \, p, 1)$ and $h(p) = \min( \b \, p, 1)$, in which
$\a>\b>1$ so that the buyer is more risk averse than the regulator.  Also, suppose the loss $X$
is exponentially distributed with mean equal to $1/\mu$. Then, for a given Lagrange multiplier
$\la \ge 0$, we find $f^\la$ to minimize
\begin{align}\label{eq:ex1}
&\int_0^{{1 \over \mu} \ln \b} [-(1 + b) e^{\mu t} + \la  e^{\mu t} + (1 + \theta) ] \, e^{-\mu t} \, df(t) \nonumber \\
&+ \int_{{1 \over \mu} \ln \b}^{{1 \over \mu} \ln \a} [-(1 + b) e^{\mu t} + \la \b + (1 + \theta)] \, e^{-\mu t} \, df(t)  \\
&+ \int_{{1 \over \mu} \ln \a}^\infty [-(1 + b) \a + \la \b + (1 + \theta)] e^{-\mu t} \, df(t). \nonumber
\end{align}
From \eqref{eq:ex1}, we consider the following cases:

\medskip

\noi {\bf Case 1:}  If $-(1 + b) + \la + (1 + \theta) = \la + \theta - b \le 0$, then all the integrands in \eqref{eq:ex1} are negative, which implies that $f^\la (x) = x$.   If $B \ge  (1+ \ln \b)/\mu = H_h(X)$, the constraint is not binding, and full insurance $f^\la$ is optimal.  Else, if $B < H_h(X)$, then the constraint binds. which implies that full insurance cannot be optimal.

\medskip

\noi {\bf Case 2:}  If $ \la + \theta - b > 0$  and $-(1 + b) \b + \la \b + (1 + \theta) \le 0$, that is,
\begin{align}\label{eq:la-case2}
b-\th < \la  \le (1+b) - (1+\th)/\b,
\end{align}
then $f^\la(x) = (x - d)_+$ for some deductible $d \in [0, (\ln \b)/\mu]$. Specifically, $d = (1/\mu) \ln\left({1 + \theta \over 1 + b - \la} \right)$.  In this case, we have $H_h((X- d)_+) = (1 + \ln \b)/\mu - d$. We have two subcases to consider.

\smallskip

\hangindent 30 pt {\bf a:}  If $\mu B \ge 1+ \ln \left( \frac{\b(1+b)}{1+\th} \right)$, then the constraint does not bind (that is, $\la = 0$), which implies that $d=(1/\mu) \ln \left( \frac{1+\th}{1+b} \right)>1/\mu$.  For this to happen, we also need to satisfy \eqref{eq:la-case2} which reduces to $0+\theta - b > 0$ and $-(1+b)\b + (1+\th) \le 0$, or equivalently, $b < \th \le (1+b)\b -1$.

\smallskip

\hangindent 30 pt {\bf b:}  Else,  if $\mu B < 1+ \ln \left( \frac{\b(1+b)}{1+\th} \right)$, then the constraint binds, and we have $\la > 0$.  Specifically $\la = (1+b) - \frac{1+\th}{\b} \, e^{\mu B -1}$. To satisfy \eqref{eq:la-case2}, we need $\mu B \ge 1$ and $\mu B < 1 + \ln \b$.   Recall that $\mu B < 1+ \ln \left( \frac{\b(1+b)}{1+\th} \right)$ in this case; by comparing the latter two upper bounds on $\mu B$, we find that Case 2b occurs if (2b1) $\th \le b$ and $1 \le \mu B < 1+\ln \b$; or if (2b2) $\th > b$ and $1 \le \mu B < 1+ \ln \left( \frac{\b(1+b)}{1+\th} \right)$.
Finally, $d=-B+(1 + \ln \b)/\mu>0$ in either situation.

\medskip

\noi {\bf Case 3:} If $ \la + \theta - b > 0$, $-(1 + b) \b + \la \b + (1 + \theta) > 0$, and $-(1 + b) \a + \la \b + (1 + \theta) < 0$, that is,
\begin{align}\label{eq:la-case3}
b-\th < \la  \hbox{ and } (1+b)\beta - (1+\theta) < \la \beta  < (1+b)\alpha - (1+\theta),
\end{align}
then $f^\la(x) = (x - d)_+$ for some deductible $d \in [ (\ln \b)/\mu, (\ln \a)/\mu]$. Specifically, $d =  (1/\mu) \ln \left( {\la \b + (1 + \theta) \over 1 + b} \right)$.  In this case, we have $H_h((X- d)_+) = \b e^{-\mu d}/\mu = {\b (1 + b) \over \mu (\la \b + (1 + \theta)) }$.  We have two subcases to consider.

\smallskip

\hangindent 30 pt {\bf a:}  If $\mu B \ge  \frac{\b(1+b)}{1+\th}$, then the constraint does not bind, and we have $\la = 0$ and $d = (1/\mu) \ln \left( \frac{1 + \theta}{1 + b} \right)$.  To satisfy \eqref{eq:la-case3}, we require $b < \th$ and $\b(1+b) < (1+\th) < \a(1+b)$.  Summarizing, Case 3a occurs if  $b < \th$, $\b(1+b)-1 < \th < \a(1+b)-1$,
and $\mu B \ge \frac{\b(1+b)}{1+\th}$.

\smallskip

\hangindent 30 pt {\bf b:}  If $\mu B <  \frac{\b(1+b)}{1+\th}$, then the constraint binds, and we have $\la = {1+b \over \mu B} - {1 + \theta \over \b}$ and $d = (1/\mu) \ln \left(\frac{\b}{\mu B} \right)$.  To satisfy  \eqref{eq:la-case3},  we require $\b/\a < \mu B < 1$ and $\mu B \left( (b-\th)\b + (1+\th) \right) < (1+b)\b$.  By considering possible values of $\theta$ and comparing with the above bounds, we find that Case 3b occurs when
$$
\left\{
\begin{aligned}
\th \le (1+b)\b - 1, && \text{ and } && \b/\a < \mu B < 1; \; \; \; \; \; \hbox{ or } \\
(1+b)\b-1 < \th < (1+b)\a-1, && \text{ and } && \b/\a < \mu B < \frac{(1+b)\b}{1+\th}.
\end{aligned}
\right.
$$

\medskip

\noi {\bf Case 4:} If $-(1 + b) \a + \la \b + (1 + \theta) = 0$, then $\la = ((1 + b) \a - (1 +
\th))/\b$, from which it follows that $ \la + \theta - b > 0$ and $-(1 + b) \b + \la \b + (1 +
\theta) > 0$.  Thus, the first two integrals in \eqref{eq:ex1} are positive, which implies that
$(f^\la)'(t) = 0$ for $t \le (\ln \a)/\mu$.  Moreover, the third integral is identically zero,
so we have infinitely many possible solutions $f^\la$. This degeneracy arises due to the
piecewise linear form of the AVaR distortions we selected.  Within this framework, we have two
subcases to consider.

\smallskip

\hangindent 30 pt {\bf a:}  If $\th = (1 + b) \a - 1$, then $\la = 0$, and the constraint does
not bind necessarily.  Thus, $f^\la$ is given by $(f^\la)'(t) = 0$ for $t \le (\ln \a)/\mu$ and
arbitrary $(f^\la)'(t) \in [0, 1]$ for $t > (\ln \a)/\mu$ such that $H_h(f^\la(X)) \le B$.

\smallskip

\hangindent 30 pt {\bf b:} If  $\th < (1 + b) \a - 1$, then $\la > 0$, and the constraint
binds.  Thus, $f^\la$ is given by $(f^\la)'(t) = 0$ for $t < (\ln \a)/\mu$ and arbitrary
$(f^\la)'(t) \in [0, 1]$ for $t \ge (\ln \a)/\mu$ such that $H_h(f^\la(X)) = B$.  We give some
examples to illustrate possible indemnity functions $f^\la$:

\smallskip

\qquad \hangindent 50 pt {\bf i:}  Let $f^\la(x) = (x - d)_+$ with deductible $d = (1/\mu) \ln \left( {\b \over \mu B} \right)$.  Note that $d \ge (\ln \a)/\mu$ if and only if $\mu B \le \b/\a$.

\smallskip

\qquad \hangindent 50 pt {\bf ii:}  Let $f^\la(x) = r(x - (\ln \a)/\mu)_+$ with proportional coverage $r = \mu B \a/\b$.  Note that $r \in [0, 1]$ if and only if $\mu B \le \b/\a$.

\smallskip

\qquad \hangindent 50 pt {\bf iii:}  Let $f^\la(x) = \min \left( r'(x - d')_+, \; m -d' \right)$ with $d'$ and $r'$ given such that $d ' \ge (\ln \a)/\mu$ and $\mu B e^{\mu d'}/\b < r' \le 1$, from which it follows that $m = (1/\mu) \ln \left({ \b r' e^{\mu d'} \over \b r' - \mu B e^{\mu d'}} \right) > d'$.

\medskip

\noi {\bf Case 5:} If $ \la + \theta - b > 0$ and $-(1 + b) \a + \la \b + (1 + \theta) > 0$, then $f^\la \equiv 0$ because all three integrals in \eqref{eq:ex1} are positive.  In this case, the constraint does not bind, and we necessarily have
$\la = 0$.  Thus, if $\theta >  (1 + b) \a - 1$, then $f^* = f^\la \equiv 0$ is optimal; that is, any amount of insurance is too expensive relative to the benefit that the buyer obtains from it.

\medskip

See Table 1 for a summary of these results as a function of the risk loading parameter $\theta$ and regulator's constraint $B$.

\begin{table}\label{table:ex}
$$
\begin{array}{|cccc|} \hline
\multicolumn{4}{|c|}{ {\th > (1+b)\a -1}} \\
\hline B > 0 & \text{Case 5 } & d = +\infty & \la = 0 \\ \hline \hline
\multicolumn{4}{|c|}{ {\th = (1+b)\a -1}} \\
\hline B > 0 & \text{Case 4a } & \text{non-unique optimum} & \la = 0 \\ \hline \hline
\multicolumn{4}{|c|}{ {(1+b)\b-1 \le \th < (1+b)\a -1}} \\
\hline \mu B \le \b/\a & \text{Case 4b } & \text{non-unique optimum} & \la = ((1 + b) \a - (1 + \th))/\b \\
\b/\a < \mu B < \frac{(1+b)\b}{1+\th} & \text{Case 3b} & d=(1/\mu) \ln \left(\frac{\b}{\mu B} \right) &
\la=\frac{1+b}{\mu B} - \frac{1+\th}{\b} > 0 \\
\mu B \ge \frac{(1+b)\b}{1+\th} & \text{Case 3a} & d= (1/\mu) \ln \left(\frac{1 + \theta}{1 + b} \right) & \la = 0 \\
\hline \hline
\multicolumn{4}{|c|}{ {b < \th < (1+b)\b -1}} \\
\hline \mu B \le \b/\a & \text{Case 4b } & \text{non-unique optimum} & \la = ((1 + b) \a - (1 + \th))/\b \\
\b/\a < \mu B < 1 & \text{Case 3b} & d=(1/\mu) \ln \left(\frac{\b}{\mu B} \right) & \la=\frac{1+b}{\mu B}
- \frac{1+\th}{\b} > 0 \\
1 \le \mu B < 1+ \ln ( \frac{\b(1+b)}{1+\th} ) & \text{ Case 2b2 } & d=-B+\frac{1+\ln \b}{\mu}&
\la = (1+b)-\frac{1+\th}{\b} \, e^{\mu B-1} \\
\mu B \ge 1+ \ln \left( \frac{\b(1+b)}{1+\th} \right) & \text{ Case 2a } & d= 1/\mu \ln \left(\frac{1+\th}{1+b} \right)
& \la=0 \\ \hline \hline \multicolumn{4}{|c|}{ {\th \le b}} \\
 \hline \mu B \le \b/\a & \text{Case 4 } & \text{non-unique optimum} & \la = ((1 + b) \a - (1 + \th))/\b \\
 \b/\a < \mu B < 1 & \text{Case 3b} & d = (1/\mu) \ln \left(\frac{\b}{\mu B} \right) & \la
= \frac{1+b}{\mu B} - \frac{1+\th}{\b} > 0 \\
1 \le \mu B < 1+\ln \b & \text{Case 2b1}  & d=-B+\frac{1+\ln \b}{\mu} & \la =
(1+b)-\frac{1+\th}{\b} \, e^{\mu
B-1} \\
\mu B \ge 1+\ln \b & \text{Case 1} & d=0 & \la=0 \\ \hline
\end{array}$$
\caption{Classification of Pareto optimal allocations of example  in Section \ref{sect:buyer}.\label{table:ex1}}
\end{table}

\section{Summary and Conclusions}\label{sect:conclusion}

In this paper, we proved that (Pareto) optimal risk sharing contracts take the form of
deductible insurance in the setting of agents endowed with distortion risk measures and linear
transaction/premium costs. Such results continue to hold under third-party constraints. This
conforms to real-life insurance contracts both in a two-agent case (for example, casualty
reinsurance) and in a  multi-agent setting (credit derivatives based on tranches).

It would be interesting to extend our results to more general setting, in particular indifference measures based on Rank Dependent Expected Utility (RDEU, also known as Maximin Expected Utility and Savage preferences). A tractable example is the exponential-distortion risk measure, see \cite{TsanakasDesli03}:
\begin{align}\label{eq:tsanakas}
H(X) = \frac{1}{\gamma}\ln \left\{\int_{-\infty}^0 \left( g[S_{\e^{\gamma Y}}(t)] - 1 \right)
\, dt + \int_0^\infty g[S_{\e^{\gamma Y}}(t)] \, dt \right\}.
\end{align}
Note that $H$ is similar to \eqref{eq:2.1} but also features the exponential utility $u(x) =
-\e^{-\gamma x}$. The preferences induced by $H$ can be seen in the context of robust utility,
where the parameter $\gamma$ is interpreted as the risk aversion coefficient, while the
distortion function $g$ corresponds to ambiguity-aversion.

One can show that $H$ is a law-invariant, convex risk measure. However, compared to our model,
$H$ is no longer coherent or comonotone additive. Nevertheless, by Remark \ref{keyRemark} our
analysis up to Theorem \ref{thm:Po} still applies. However, because the non-linear
$\log$-transformation in \eqref{eq:tsanakas} is global, the structure of Theorem \ref{thm:Po}
does not hold because we can no longer perform $t$-by-$t$ optimization for the optimal risk
allocation $f$.

From a general viewpoint, our work confirms previous results of Jouini et al.~\cite{JST} (and
originally Arrow~\cite{Arrow63}) on optimality of deductible insurance. Conversely, it
contrasts with possibility of proportional risk sharing obtained in Barrieu and El
Karoui~\cite{BarrieuKaroui05} (and originally Borch~\cite{Borch62}). The key step in our method
relies on comonotonicity of Pareto optimal allocations due to the consistency of preferences
with the stochastic convex order $\le_{cx}$. Thus, we raise the conjecture that in the setting
of law-invariant convex risk measures, optimal risk sharing always leads to insurance that
incorporates a ladder of deductibles (both in unconstrained and constrained settings).

\end{document}

%% file: constrainedDed2.pdftex_t
\setlength{\unitlength}{4144sp}
\begin{picture}(0,0)(-70,-70)
\includegraphics{constrainedDed2.pdf}%
\end{picture}

\setlength{\unitlength}{4144sp}%
\begingroup\makeatletter\ifx\SetFigFontNFSS\undefined%
\gdef\SetFigFontNFSS#1#2#3#4#5{%
  \reset@font\fontsize{#1}{#2pt}%
  \fontfamily{#3}\fontseries{#4}\fontshape{#5}%
  \selectfont}%
\fi\endgroup%
\begin{picture}(9705,7556)(436,-7111)
\put(676,-61){\makebox(0,0)[lb]{\smash{{\SetFigFontNFSS{16}{14.4}{\rmdefault}{\mddefault}{\updefault}{\color[rgb]{0,0,0}0}%
}}}}
\put(9001,164){\makebox(0,0)[lb]{\smash{{\SetFigFontNFSS{16}{14.4}{\rmdefault}{\mddefault}{\updefault}{\color[rgb]{0,0,0}1}%
}}}}
\put(5401,164){\makebox(0,0)[lb]{\smash{{\SetFigFontNFSS{16}{14.4}{\rmdefault}{\mddefault}{\updefault}{\color[rgb]{0,0,0}$1/\beta$}%
}}}}
\put(500,-7036){\makebox(0,0)[lb]{\smash{{\SetFigFontNFSS{16}{14.4}{\rmdefault}{\mddefault}{\updefault}{\color[rgb]{0,0,0}$-1$}%
}}}}
\put(4726,164){\makebox(0,0)[lb]{\smash{{\SetFigFontNFSS{16}{14.4}{\rmdefault}{\mddefault}{\updefault}{\color[rgb]{0,0,0}$1/\alpha_2$}%
}}}}
\put(7651,164){\makebox(0,0)[lb]{\smash{{\SetFigFontNFSS{16}{14.4}{\rmdefault}{\mddefault}{\updefault}{\color[rgb]{0,0,0}$1/\alpha_1$}%
}}}}
\put(10126,-61){\makebox(0,0)[lb]{\smash{{\SetFigFontNFSS{16}{14.4}{\rmdefault}{\mddefault}{\updefault}{\color[rgb]{0,0,0}$p$}%
}}}}
\put(6076,164){\makebox(0,0)[lb]{\smash{{\SetFigFontNFSS{16}{14.4}{\rmdefault}{\mddefault}{\updefault}{\color[rgb]{0,0,0}$p^*$}%
}}}}
\put(3000,-5000){\makebox(0,0)[lb]{\smash{{\SetFigFontNFSS{16}{14.4}{\rmdefault}{\mddefault}{\updefault}{\color[rgb]{0,0,0}$p^*_2 = \frac{\la(1+\th) - \th + b_1}{(1+\th)(b_1+\la) - (1+b_1)\a_1 \th}$}%
}}}}
\end{picture}%

%% file: constrainedDed1.pdftex_t
\setlength{\unitlength}{4144sp}
\begin{picture}(0,0)(-70,-70)
\includegraphics{constrainedDed1.pdf}%
\end{picture}
\setlength{\unitlength}{4144sp}%
\begingroup\makeatletter\ifx\SetFigFontNFSS\undefined%
\gdef\SetFigFontNFSS#1#2#3#4#5{%
  \reset@font\fontsize{#1}{#2pt}%
  \fontfamily{#3}\fontseries{#4}\fontshape{#5}%
  \selectfont}%
\fi\endgroup%
\begin{picture}(9480,7556)(661,-7111)
\put(676,-61){\makebox(0,0)[lb]{\smash{{\SetFigFontNFSS{16}{14.4}{\rmdefault}{\mddefault}{\updefault}{\color[rgb]{0,0,0}0}%
}}}}
\put(9001,164){\makebox(0,0)[lb]{\smash{{\SetFigFontNFSS{16}{14.4}{\rmdefault}{\mddefault}{\updefault}{\color[rgb]{0,0,0}1}%
}}}}
\put(6301,164){\makebox(0,0)[lb]{\smash{{\SetFigFontNFSS{16}{14.4}{\rmdefault}{\mddefault}{\updefault}{\color[rgb]{0,0,0}$1/\beta$}%
}}}}
\put(500,-7036){\makebox(0,0)[lb]{\smash{{\SetFigFontNFSS{16}{14.4}{\rmdefault}{\mddefault}{\updefault}{\color[rgb]{0,0,0}$-1$}%
}}}}
\put(4501,164){\makebox(0,0)[lb]{\smash{{\SetFigFontNFSS{16}{14.4}{\rmdefault}{\mddefault}{\updefault}{\color[rgb]{0,0,0}$1/\alpha_2$}%
}}}}
\put(7651,164){\makebox(0,0)[lb]{\smash{{\SetFigFontNFSS{16}{14.4}{\rmdefault}{\mddefault}{\updefault}{\color[rgb]{0,0,0}$1/\alpha_1$}%
}}}}
\put(10126,-61){\makebox(0,0)[lb]{\smash{{\SetFigFontNFSS{16}{14.4}{\rmdefault}{\mddefault}{\updefault}{\color[rgb]{0,0,0}$p$}%
}}}}
\put(5401,164){\makebox(0,0)[lb]{\smash{{\SetFigFontNFSS{16}{14.4}{\rmdefault}{\mddefault}{\updefault}{\color[rgb]{0,0,0}$p^*$}%
}}}}
\put(3000,-5000){\makebox(0,0)[lb]{\smash{{\SetFigFontNFSS{16}{14.4}{\rmdefault}{\mddefault}{\updefault}{\color[rgb]{0,0,0}$p^*_1 = \frac{\th -b_1-\la}{\th(1+b_1)\a_1 - b_1(1+\th) + \la[\th\beta - (1+\th)]}$}%
}}}}
\end{picture}%

%% file: constrainedCapDed.pdftex_t
\setlength{\unitlength}{4144sp}
\begin{picture}(0,0)(-70,-70)
\includegraphics{constrainedCapDed.pdf}%
\end{picture}

\setlength{\unitlength}{4144sp}%
\begingroup\makeatletter\ifx\SetFigFontNFSS\undefined%
\gdef\SetFigFontNFSS#1#2#3#4#5{%
  \reset@font\fontsize{#1}{#2pt}%
  \fontfamily{#3}\fontseries{#4}\fontshape{#5}%
  \selectfont}%
\fi\endgroup%
\begin{picture}(9480,7556)(661,-7111)
\put(676,-61){\makebox(0,0)[lb]{\smash{{\SetFigFontNFSS{16}{14.4}{\rmdefault}{\mddefault}{\updefault}{\color[rgb]{0,0,0}0}%
}}}}
\put(9001,164){\makebox(0,0)[lb]{\smash{{\SetFigFontNFSS{16}{14.4}{\rmdefault}{\mddefault}{\updefault}{\color[rgb]{0,0,0}1}%
}}}}
\put(5176,164){\makebox(0,0)[lb]{\smash{{\SetFigFontNFSS{16}{14.4}{\rmdefault}{\mddefault}{\updefault}{\color[rgb]{0,0,0}$1/\beta$}%
}}}}
\put(500,-7036){\makebox(0,0)[lb]{\smash{{\SetFigFontNFSS{16}{14.4}{\rmdefault}{\mddefault}{\updefault}{\color[rgb]{0,0,0}$-1$}%
}}}}
\put(6526,164){\makebox(0,0)[lb]{\smash{{\SetFigFontNFSS{16}{14.4}{\rmdefault}{\mddefault}{\updefault}{\color[rgb]{0,0,0}$1/\alpha_2$}%
}}}}
\put(7651,164){\makebox(0,0)[lb]{\smash{{\SetFigFontNFSS{16}{14.4}{\rmdefault}{\mddefault}{\updefault}{\color[rgb]{0,0,0}$1/\alpha_1$}%
}}}}
\put(7201,164){\makebox(0,0)[lb]{\smash{{\SetFigFontNFSS{16}{14.4}{\rmdefault}{\mddefault}{\updefault}{\color[rgb]{0,0,0}$p_2$}%
}}}}
\put(5851,164){\makebox(0,0)[lb]{\smash{{\SetFigFontNFSS{16}{14.4}{\rmdefault}{\mddefault}{\updefault}{\color[rgb]{0,0,0}$p_1$}%
}}}}
\put(10126,-61){\makebox(0,0)[lb]{\smash{{\SetFigFontNFSS{16}{14.4}{\rmdefault}{\mddefault}{\updefault}{\color[rgb]{0,0,0}$p$}%
}}}}
\end{picture}%